\documentclass[]{article}

\usepackage{tikz}
\usepackage{graphicx}
\usepackage{subfigure}
\usepackage{booktabs}
\usepackage{hyperref}

\usepackage{times}

\usepackage{fancyhdr}
\usepackage{color}
\usepackage{tikz}
\usepackage{algorithm}
\usepackage{algorithmic}
\usepackage{natbib}
\usepackage{authblk}

\newcommand{\ind}[1]{\mathbb I(#1)}
\newcommand{\eig}{\mathtt{eig}}
\newcommand{\eps}{\varepsilon}
\newcommand{\R}{\mathbb R}
\newcommand{\M}{\mathcal M}
\newcommand{\Y}{\mathcal Y}
\newcommand{\Z}{\mathcal Z}
\newcommand{\dedge}{d_\text{edge}}
\newcommand{\dnode}{d_\text{node}}
\newcommand{\Meps}{\M_\eps}
\newcommand{\deps}{d_\eps}
\newcommand{\depsMeps}{\deps \circ \Meps}
\newcommand{\nmin}{n_\text{min}}
\newcommand{\nmax}{n_\text{max}}
\newcommand{\nmintilde}{\tilde n_\text{min}}
\newcommand{\SBM}{\mathrm{SBM}}
\newcommand{\SSBM}{\mathrm{SSBM}}
\newcommand{\DCBM}{\mathrm{DCBM}}
\newcommand{\SDCBM}{\mathrm{SDCBM}}
\newcommand{\diag}{\mathrm{diag}}
\newcommand{\erdosrenyi}{Erd\H{o}s--R\'enyi}

\usepackage[utf8]{inputenc}
\usepackage{amssymb,amsmath,amsthm}
\newtheorem{definition}{Definition}
\newtheorem{theorem}{Theorem}[section]

\newtheorem{lemma}[theorem]{Lemma}
\newtheorem{corollary}[theorem]{Corollary}
\newtheorem{fact}[theorem]{Fact}

\title{Consistent Spectral Clustering of Network Block Models under Local Differential Privacy}
\author{Jonathan Hehir}
\author{Aleksandra Slavkovi{\'c}}
\author{Xiaoyue Niu}
\affil{Department of Statistics, Pennsylvania State University, University Park, PA}
\date{\today}

\begin{document}

\maketitle

\begin{abstract}
     The stochastic block model (SBM) and degree-corrected block model (DCBM) are network models often selected as the fundamental setting in which to analyze the theoretical properties of community detection methods. We consider the problem of spectral clustering of SBM and DCBM networks under a local form of edge differential privacy. Using a randomized response privacy mechanism called the edge-flip mechanism, we develop theoretical guarantees for differentially private community detection, demonstrating conditions under which this strong privacy guarantee can be upheld while achieving spectral clustering convergence rates that match the known rates without privacy. We prove the strongest theoretical results are achievable for dense networks (those with node degree linear in the number of nodes), while weak consistency is achievable under mild sparsity (node degree greater than $\sqrt{n}$).  We empirically demonstrate our results on a number of network examples.
\end{abstract}

\section{Introduction}
\label{introduction}

In the field of network data analysis, a common problem is community detection, the algorithmic discovery of clusters or communities of dense
connection. There exist numerous parametric network models that exhibit community structure. Two fundamental models used in the analysis of community detection algorithms are the stochastic block model (SBM), first introduced in \citet{holland1983stochastic}, and the degree-corrected block model (DCBM) of \citet{karrer2011stochastic}, a popular SBM extension that allows for networks with more realistic heterogeneity in node degree.

We consider the problem of estimating latent community structure in networks using SBM and DCBM via the method of spectral clustering. In the absence of privacy, this problem has gained popularity, as spectral clustering has been shown to be a computationally tractable method for group estimation with satisfying theoretical guarantees \citep{mcsherry2001spectral,ng2002spectral,rohe2011spectral,qin2013regularized, lei2015consistency,joseph2016impact,binkiewicz2017covariate,abbe_community_2018,abbe2021ellp}. In this setting, the clusters of a given network are informed by the relationships contained within that network. These relationships are represented by the edges of the network, which may reflect friendships, partnerships, collaborations, communications, financial transactions, similarities, or other interactions between some set of entities, say, people or businesses. In many cases, these relationships may constitute sensitive or confidential information, and so we may desire a clustering that reveals high-level structures in the network while preserving the privacy of low-level relationships.

To provide such a privacy guarantee, we turn to a form of local differential privacy (DP) for networks (see, e.g.,  \cite{imola2021locally} for networks, and \cite{duchi2013local} for more general ideas on local DP). We assume that the nodes of the network are known but that their relationships are sensitive. We protect these relationships through a randomized-response mechanism applied to the edges of the network. Previous works, such as  \cite{mulle2015privacy} and \cite{karwa_sharing_2017}, have used similar techniques to obtain a synthetic network that satisfies $\eps$-DP \citep[of][]{dwork2006calibrating} with respect to the edges of the network. Recent works \citep{qin_generating_2017,imola2021locally} have emphasized that these synthetic networks can be constructed in a distributed fashion, requiring that no single party has knowledge of the full set of true network relationships. We construct our synthetic networks in such a manner, then apply a modified spectral clustering algorithm to obtain consistent estimates of group membership for SBM and DCBM networks.

The problem of performing DP community detection has previously been given empirical consideration  \citep{mulle2015privacy,nguyen2016detecting,qin_generating_2017}, but to our knowledge, our analysis of the cost of privacy in this setting is the first theoretical treatment of the subject. Our work generalizes the non-private results of \citet{lei2015consistency} to the DP setting, and we demonstrate certain conditions under which our DP estimator matches the known non-private convergence rates. We attribute these results to certain desirable properties of the edge-flip mechanism that we explore in some detail. While this method performs well for dense networks, the magnitude of error introduced by the local DP mechanism results in a slowing of the convergence rate for sparse networks and a requirement that node degree grows faster than $\sqrt{n}$. We derive these results in the context of more general finite-sample and asymptotic bounds on misclassification rates with privacy.

The paper is structured as follows: We first introduce the network models and notation in Section~\ref{models-and-notation}. In Section~\ref{privacy-mechanism}, we review differential privacy and define the edge-flip mechanism, and we demonstrate that the edge-flip mechanism can be viewed as a mixture distribution (Lemma~\ref{thm:mixture}), leading to two corollaries: a closure property for SBM and a more general post-processing step that allows us to preserve the expected spectral embedddings of edge-flipped networks. In Section~\ref{clustering-methods}, we propose a modified spectral clustering algorithm based on this method---with concentration bounds for the estimation of spectral embeddings of SBM and DCBM networks proven in Lemma~\ref{thm:frobenius}. Our main theoretical results on finite-sample bounds on misclassification and asymptotic convergence rates for private spectral clustering are presented in Section~\ref{theoretical-results}. Proofs of these results are deferred to Section~\ref{proofs}. In Section~\ref{empirical-evaluation}, we evaluate the performance of private spectral clustering on both simulated and observed networks. Finally, we conclude with a discussion of limitations and open questions in Section~\ref{discussion}.

\section{Network Models and Notation}
\label{models-and-notation}

SBM and DCBM are models for binary, undirected networks
without covariates. Edges in both networks occur as independent Bernoulli
random variables. Each of the $n$ nodes is assigned to one of $k$ communities, or \textit{blocks}, and these memberships are denoted in the parameter $\theta \in [k]^n$ (where $[k] = \{1, \dots, k\}$). In SBM, the probability of a given edge occurring between nodes $i$ and $j$ depends only on the blocks to which $i$ and $j$ belong, with those block-to-block edge probabilities recorded in the symmetric matrix $B \in (0, 1]^{k \times k}$. In DCBM, edge probabilities further depend on a parameter $\psi \in (0, 1]^n$ that determines the relative expected node degree for each node. If we let $Y$ be the symmetric $n \times n$ adjacency matrix from a DCBM, the elements of $Y$ are then distributed:
$$
\begin{aligned}
Y_{ij} \overset{\text{ind}}{\sim} \begin{cases}
    \mathrm{Bernoulli}(\psi_i \psi_j B_{\theta_i \theta_j}), & i < j \\
    0, & i = j \\
    Y_{ji}, & i > j \\
\end{cases}
\end{aligned}
$$

For such a network, we will write $Y \sim \DCBM(\theta, \psi, B)$. The stochastic block model can be regarded as a special case of the DCBM, where $\psi = \mathbf 1_n$, a vector of $n$ ones, i.e.,
\begin{equation}
\label{eq:sbm-dcbm-relation}
\SBM(\theta, B) \overset D = \DCBM(\theta, \mathbf 1_n, B) ,
\end{equation}
where $\overset D =$ indicates equality in joint distribution of network edges.

We let $C_j = \{ i : \theta_i = j \}$ denote the set of nodes in the $j$-th community. We denote the size of the $j$-th community in an SBM or DCBM (i.e., the number of nodes in a block) as $n_j = |C_j|$, the smallest of which we denote $\nmin = \min_{j \in [k]} n_j$, the largest as $\nmax = \max_{j \in [k]} n_j$, and the second-largest as $\nmax'$. As in \citet{lei2015consistency}, we denote the effective size of the $j$-th block in a DCBM as $\tilde n_j = \sum_{i \in C_j} \psi_i^2$ and the largest effective size of a block as $\tilde n_\text{max} = \max_{j \in [k]} \tilde n_j$. As a measure of heterogeneity within the $j$-th block, we use $\nu_j = n_j^{-2}(\sum_{i \in C_j} \psi_i^{-2})(\sum_{i \in C_j} \psi_i^{2})$.

We use $\lambda_B$ to denote the smallest absolute nonzero eigenvalue of the matrix $B$. The largest entry in $B$ is denoted $\max B = \max_{ij} B_{ij}$. We represent networks via adjacency matrices, e.g., $Y \in \{0, 1\}^{n \times n}$. The $i$-th row of the matrix $Y$ is denoted $Y_{i*}$. We use $\| x \|_p$ to denote the $\ell_p$ norm of a vector $x$, $\|A\|_F$ to denote the Frobenius norm of a matrix, and $\| A \|$ to denote the operator norm of the matrix $A$, i.e., $\| A \| = \sup_{\|x\|_2 = 1} \|Ax\|_2$.

In asymptotic notation, we write $a_n = o(b_n)$ when $| a_n / b_n | \to 0$ as $n \to \infty$; $a_n = \omega(b_n)$ when $| a_n / b_n | \to \infty$ as $n \to \infty$; $a_n = O(b_n)$ when $| a_n / b_n | \leq C$ for some $C > 0$ and all $n$; $a_n = \Omega(b_n)$ when $| a_n / b_n | \geq C$ for some $C > 0$ and all $n$; and $a_n = \Theta(b_n)$ when $a_n = O(b_n)$ and $a_n = \Omega(b_n)$. Finally, we write $X_n = O_P(b_n)$ if for any $\alpha > 0$ there exists a constant $M$ such that $P(|X_n / b_n| > M) < \alpha$ for all large $n$.

\section{Privacy Mechanism as a Mixture Distribution}
\label{privacy-mechanism}

\subsection{Defining Privacy}

The framework of differential privacy offers more than just a formal privacy definition: it offers \textit{many} privacy definitions that build on each other. To fully appreciate the privacy implications of a given definition---and thereby to justify our specific choice of definition---we briefly review the core concepts of DP and discuss the nuanced distinctions between privacy definitions that could be applied to the problem at hand. Consider first the typical definition of DP:

\begin{definition}[Differential Privacy \citep{dwork2006calibrating}]
Let $\eps > 0, \delta \in [0, 1)$. Let $\Y, \Z$ be sets, and let $d : \Y \times \Y \to \mathbb Z_{\geq 0}$ be an integer-valued metric. Let $\M : \Y \to \Z$ be a randomized algorithm. We say $\M$ satisfies \textbf{$(\eps, \delta)$--differential privacy} with respect to $d$ if for any $y, y' \in \Y$ satisfying $d(y, y') = 1$ and any $E \subseteq \Z$, we have:
\begin{equation}
\label{eq:general-dp}
P(\M(y) \in E) \leq e^\eps \cdot P(\M(y') \in E) + \delta
\end{equation}
If $\delta = 0$, we say $\M$ satisfies \textbf{$\eps$--differential privacy}.
\end{definition}

Such an algorithm $\M$ is called a \textit{privacy mechanism}. The elements $y$ and $y'$ are often referred to as \textit{databases} and thought of as a collection of records from some set, say, $\Y = \R^n$ or $\Y = \{0, 1\}^n$. The metric $d$ is usually chosen to be the Hamming distance, such that (\ref{eq:general-dp}) can be interpreted thus: changing any one record in an arbitrary database $y \in \Y$ does not significantly affect the distribution of $\M(y)$.\footnote{Note that the databases $y, y'$ are not presented as random quantities in this definition. One can think of $\M(y)$ as a distribution conditioned on the observed data, as in \citet{wasserman_statistical_2010}.} Consequently, one can gain virtually no information about a single record based on the private output $\M(y)$. The parameters $\eps$ and $\delta$ quantify the strength of the privacy guarantee, with smaller values conferring a stronger guarantee. The case when $\delta > 0$ is often referred to as \textit{approximate differential privacy}, while the case where $\delta = 0$ is called \textit{pure differential privacy}. The parameter $\eps$ is commonly called the \textit{privacy-loss budget}.

In the network setting, there are two primary choices for the metric $d$ that are frequently employed, and these correspond to the notions of edge DP and node DP \citep{sahai_analyzing_2013,qin_generating_2017}:

\begin{definition}[Edge DP]
Let $Y, Y'$ be two networks with $n$ nodes. Then $\dedge(Y, Y')$ is given by the total number of edges that differ between $Y$ and $Y'$. If $\M$ satisfies $(\eps, \delta)$--DP with respect to $\dedge$, then we say $\M$ satisfies \textbf{$(\eps, \delta)$--edge differential privacy}.
\end{definition}

\begin{definition}[Node DP]
Let $Y, Y'$ be two networks with $n$ nodes. Then $\dnode(Y, Y')$ is given by the minimum number of nodes in $Y$ whose incident edges could be modified in order to obtain $Y'$. If $\M$ satisfies $(\eps, \delta)$--DP with respect to $\dnode$, then we say $\M$ satisfies \textbf{$(\eps, \delta)$--node differential privacy}.
\end{definition}

Since $\dedge(Y, Y') \geq \dnode(Y, Y')$, the definition of node DP is more strict than the definition of edge DP. Indeed, node DP implies edge DP, but the reverse is not true. In many cases, the strictness of node DP precludes meaningful analysis \citep{sahai_analyzing_2013}. This is the case, for example, in the problem we wish to solve. Suppose $\hat \theta(Y)$ is an $\eps$--DP estimator of the group membership $\theta$ for the network $Y$ with respect to a distance measure $d$. This implies that when we look at any given node's estimated label, $[\hat \theta(\cdot)]_i$, across two networks $Y, Y'$ satisfying $d(Y, Y') = 1$, we must satisfy:
$$
\frac{P( [\hat \theta(Y)]_i = \ell)}{P( [\hat \theta(Y')]_i = \ell)} \in [e^{-\eps}, e^\eps] \quad \forall \ell
$$
For edge DP ($d = \dedge$), this means that changing any single edge in the network cannot significantly affect the distribution of the estimated label for a given node. For node DP ($d = \dnode$), however, we can take any given node and change any subset (or even all) of its incident edges without significantly affecting the distribution of its estimated label. Since our goal is to infer labels from precisely these edges, this notion of privacy is too strict for our purposes. For this reason, we will focus on edge-based definitions of DP, which aim to protect the privacy of individual relationships within the network.\footnote{Edge-based definitions of DP implicitly assume that the only sensitive information in the network is encoded in these relationships. For example, in a community detection setting, we assume that the identities of nodes are not sensitive but that their relationships are.}

So far, the DP definitions we have considered all fall under the umbrella of \textit{central DP}. In central DP, we assume that one party, often called the \textit{trusted curator}, has complete knowledge of the true database $y$. This arrangement is not always desirable, which motivates the concept of \textit{local DP} \citep{duchi2013local}. In local DP, records in the database are held by a number of distributed parties (e.g., users of a social network), and only these parties require true knowledge of their records. To facilitate some central analysis of the database, a randomized algorithm is independently applied to each record to produce a \textit{local differentially private view} of that record. These local DP views may then be shared with a central processor for further analysis. Thus, in contrast with central DP, where a single mechanism is applied at the database level by a single database owner, local DP applies privacy mechanisms at the record level, eliminating the need for a trusted curator, while still allowing for centralized analysis of the data.

Applying local DP to the edge setting, we have further choices yet. Likely the most widely known definition of local DP in the edge setting is \textit{edge local DP} of \citet{qin_generating_2017}, but we will instead focus on \textit{relationship DP} as defined in \citet{imola2021locally}. This definition is more tailored to the setting of undirected networks, where it provides a stronger privacy guarantee.\footnote{More precisely, in an undirected network, $\eps$--edge local DP implies $2\eps$--relationship DP \citep{imola2021locally}.} In relationship DP (as in edge local DP), each node reports a private view (or summary) of its \textit{neighbor list}, the set of nodes with which it shares an edge. The formal definition is given below.

\begin{definition}[Relationship DP \citep{imola2021locally}]
Let $\eps > 0$, and let $\M_1, \dots, \M_n$ be randomized algorithms with domain $\{0, 1\}^n$. The algorithm $\M(Y) = (\M_1(Y_{1*}), \dots, \M_n(Y_{n*}))$ is said to satisfy \textbf{$\eps$--relationship DP} if for each $E \subseteq \mathrm{Range}(\M)$ and networks $Y, Y'$ satisfying $\dedge(Y, Y') = 1$, we have:
$$
P(\mathcal M (Y) \in E) \leq e^\varepsilon \cdot P(\mathcal M (Y') \in E)
$$
\end{definition}

\subsection{The Edge-Flip Mechanism}

The specific privacy mechanism we employ is a simple randomized-response mechanism that produces an $\eps$-relationship-DP synthetic copy $\M_\eps(Y)$ of the original network $Y$ by randomly flipping edges in $Y$. This can be performed locally by assigning each node to flip and self-report a subset of its edges. In this way, no single party ever needs full knowledge of the true adjacency matrix $Y$. Since differential privacy is closed under post-processing \citep{dwork2006calibrating}, any analysis performed on the synthetic network $\M_\eps(Y)$ preserves $\eps$--relationship DP. Edge-flipping mechanisms have been utilized in several earlier papers and interpreted under various edge DP definitions, including (central) edge DP \citep{mulle2015privacy,karwa_inference_2016, karwa_sharing_2017}, edge local DP \citep{qin_generating_2017}, and relationship DP \citep{imola2021locally}. We include here an explicit adaptation of the edge-flipping procedure from \citet{imola2021locally}, which we term the \textit{symmetric edge-flip mechanism}.

\begin{definition}[Symmetric Edge-Flip Mechanism]
Let $\eps > 0$. For $i \in [n]$, let $\M_i : \{0, 1\}^n \to \{0, 1\}^n$ such that:
$$
[\M_i(x)]_j \overset{\text{ind}}{=} \begin{cases}
  0 & i \geq j \\
  1 - x_j & i < j \quad \text{w.p.} \quad \frac{1}{1 + e^{\varepsilon}} \\
  x_j & i < j \quad \text{w.p.} \quad \frac{e^{\varepsilon}}{1 + e^{\varepsilon}}
\end{cases},
$$
and let
$$
T(Y) = \begin{bmatrix}
    \M_1(Y_{1*}) \\
    \vdots \\
    \M_n(Y_{n*})
\end{bmatrix}.
$$
Then the \textbf{$n \times n$ symmetric edge-flip mechanism} is the mechanism $\mathcal M_\varepsilon(Y) = T(Y) + [T(Y)]^T$.
\end{definition}

\begin{theorem}
\label{thm:privacy}
The symmetric edge-flip mechanism $\mathcal M_\varepsilon$ satisfies $\eps$--relationship DP.
\end{theorem}
\begin{proof}
The proof of this follows from Theorem~3 of \citet{imola2021locally} and the corresponding discussion. For completeness, we give a formal proof in Section~\ref{privacy-proof}.
\end{proof}

The simplicity of the edge-flip mechanism affords it several key advantages: In practice, it is easy and flexible to implement. In theory, the distribution of the resulting synthetic network is transparent and tractable. In fact, in Lemma~\ref{thm:mixture} we show that the network generated by the edge-flip mechanism is a mixture of the original non-private network with an \erdosrenyi{} network. This in turn leads to closure and statistical inference properties that are important and useful, both theoretically and practically. These properties are demonstrated in the corollaries that follow. 

\begin{lemma}
\label{thm:mixture}
Let $Y \in \{0, 1 \}^{n \times n}$ be a random, undirected, binary network. Let $Z \sim G(n, \frac 12)$ be an \erdosrenyi{} random graph with $n$ nodes and edge probability $\frac 12$, and let $U_{ij} \overset{\text{ind}}{\sim} \mathrm{Bernoulli}\left( \frac{2}{e^\eps + 1} \right)$ for $1 \leq i < j \leq n$. Define $\M'(Y)$ to be the symmetric mixture network such that for $i < j$:
$$
\M'(Y)_{ij} \, | \, Y, Z, U = (Z_{ij})^{U_{ij}} (Y_{ij})^{(1 - U_{ij})}
$$
Then $\M'(Y)$ is equal in distribution to $\Meps(Y)$.
\end{lemma}

\begin{proof}
It is sufficient to show that the conditional distributions of $\M'(Y) \, | \, Y$ and $\Meps(Y) \, | \, Y$ are equivalent. Conditioned on $Y$, the entries of $\M'(Y)$ and $\Meps(Y)$ are independent, binary random variables, and:
$$
\begin{aligned}
P(\M'(Y)_{ij} = 1 \, | \, Y) &= E [\M'(Y)_{ij} \, | \, Y] \\
  &= E [\; E [ \M'(Y)_{ij} | \, Y, Z, U] \; | \, Y] \\
  &= E \left[ U_{ij} Z_{ij} + (1 - U_{ij}) Y_{ij} \, \middle| \, Y \right] \\
  &= \frac {1}{e^\eps + 1} + \frac {e^\eps - 1}{e^\eps + 1} Y_{ij} \\
  &= \frac {1}{e^\eps + 1} (1 - Y_{ij}) + \frac {e^\eps}{e^\eps + 1} Y_{ij} \\
  &= E [ \M_\eps(Y)_{ij} \, | \, Y] \\
  &= P(\M_\eps(Y)_{ij} = 1 \, | \, Y)
\end{aligned}
$$
\end{proof}

Lemma~\ref{thm:mixture} provides useful intuition about the edge-flip mechanism. As $\eps \to \infty$, our synthetic network $\Meps(Y)$ approaches the true network $Y$, and as $\eps \to 0$, we approach an \erdosrenyi{} network. For block models, this is a welcome property, as it means the community structure of the network is exactly preserved. For the SBM in particular, we have closure of the model family under $\M_\eps$. The following corollary follows from a simple extension of the proof of Lemma~\ref{thm:mixture}.

\begin{corollary}
\label{thm:closure}
If $Y \sim \SBM(\theta, B)$, then $\M_\eps(Y) \sim \SBM(\theta, \tau_\eps(B))$, where:
\begin{equation}
\label{eq:tau-eps}
\tau_\eps(B) = \frac 1{e^\eps + 1} \mathbf 1_k \mathbf 1_k^T + \frac {e^\eps - 1}{e^\eps + 1} B.
\end{equation}
\end{corollary}
% \begin{proof}
% Observe that marginally, $\Meps(Y)_{ij}$ is an independent, binary random variable, as each entry in $Y$ is independent, and the edge flips are performed independently. From the proof of Lemma~\ref{thm:mixture}, we can see that for $i < j$:
% $$
% \begin{aligned}
% P(\M_\eps(Y)_{ij} = 1) &= E[ \; E[ \M_\eps(Y)_{ij} \, | \, Y ] \; ] \\
%     &= E \left[ \frac {1}{e^\eps + 1} + \frac{e^\eps - 1}{e^\eps + 1} Y_{ij} \right] \\
%     &= \frac {1}{e^\eps + 1} + \frac{e^\eps - 1}{e^\eps + 1}B_{\theta_i \theta_j}
% \end{aligned}
% $$
% Thus $\M_\eps(Y)_{ij} \sim \mathrm{Bernoulli} \left(\frac {1}{e^\eps + 1} + \frac{e^\eps - 1}{e^\eps + 1} B_{\theta_i \theta_j} \right)$.
% \end{proof}

Closure of a particular family of models under edge flipping can be used to obtain theoretical guarantees for differentially private procedures via direct application of non-private theory. SBM is not alone in holding this convenient property, with some extensions, such as the mixed-membership SBM of \citet{airoldi2008mixed}, sharing a similar closure property. Unfortunately, not all SBM extensions afford such a closure property. In particular, DCBM is not closed under edge-flipping. For this reason, we will explore a more generally applicable framework here, relying on a weaker property of the edge-flip mechanism that holds for any random binary network: via a small ``downshift" transformation to an edge-flipped network, we can recover a network whose expectation matches that of the original network, up to a scaling factor.

\begin{corollary}
\label{thm:scaledexpectation}
Let $Y \in \{0, 1 \}^{n \times n}$ be a random, undirected, binary network, and let $(\depsMeps)(Y) = \Meps(Y) - (e^\eps + 1)^{-1}(\mathbf 1_n \mathbf 1_n^T - \mathbf I_n)$ be the downshifted edge-flipped network. Then:
$$
E (\depsMeps)(Y) = \frac{e^\eps - 1}{e^\eps + 1} EY.
$$
\end{corollary}

\begin{proof}
This follows from a slight rewriting of the derivation in the proof of Lemma~\ref{thm:mixture}, namely:
$$
\begin{aligned}
E[\Meps(Y)] &= E[ \; E[ \M_\eps(Y) \, | \, Y ] \; ] \\
    &= \frac {1}{e^\eps + 1} (\mathbf 1_n \mathbf 1_n^T - \mathbf I_n) + \frac{e^\eps - 1}{e^\eps + 1} EY
\end{aligned}
$$
\end{proof}

Lemma~\ref{thm:scaledexpectation} has important implications: for any random, undirected, binary network $Y$, the matrices $E[Y]$ and $E[(\depsMeps)(Y)]$ share the same eigenvectors. This fact leads to a more general post-processing step under DP that allows us to preserve the expected spectral embeddings of edge-flipped networks.

\section{Modified Clustering Methods with Concentration Bounds}
\label{clustering-methods}

Our goal is to use spectral clustering to estimate the unobserved block membership parameter $\theta$ while satisfying $\eps$-relationship DP. A variety of related clustering algorithms exist under the umbrella term of \textit{spectral clustering} \citep{ng2002spectral,von2007tutorial,rohe2011spectral,lei2015consistency,binkiewicz2017covariate}. In our setting, we focus on the method of adjacency spectral clustering with $k$ known, following the framework of \citet{lei2015consistency}. In adjacency spectral clustering, we place the leading $k$ eigenvectors (by absolute value of corresponding eigenvalues) of the observed adjacency matrix $Y$ into an $n \times k$ matrix $\hat U$. We consider each row of $\hat U$ to be the spectral embedding of the respective node in the network, and we perform a clustering process over these embeddings: for SBM, we perform simple $k$-means, and for DCBM, we normalize the embeddings to have unit norm before performing $k$-medians clustering. The resulting estimated cluster memberships serve as our estimator, $\hat \theta$.

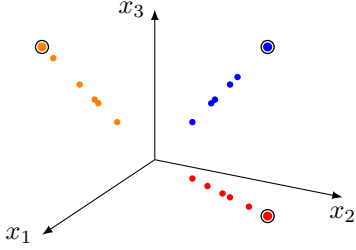
\begin{figure}
    \centering
    \begin{tikzpicture}[scale=.5]
    \draw[latex-] (0, 4) node[left]{$x_3$}
            -- (0, 0);
    \draw[latex-] (-3, -2) node[left]{$x_1$}
            -- (0, 0);
    \draw[latex-] (5, -1) node[below]{$x_2$}
            -- (0, 0);
    \filldraw[color=black] (3, 3) circle[radius=5pt]
        (-3, 3) circle[radius=5pt]
        (3, -1.5) circle[radius=5pt];
    \filldraw[color=white] (3, 3) circle[radius=4pt]
        (-3, 3) circle[radius=4pt]
        (3, -1.5) circle[radius=4pt];
    \filldraw[color=blue] (1, 1) circle[radius=2pt]
        (1.5, 1.5) circle[radius=2pt]
        (1.6, 1.6) circle[radius=2pt]
        (2, 2) circle[radius=2pt]
        (2.2, 2.2) circle[radius=2pt]
        (3, 3) circle[radius=3pt];
    \filldraw[color=orange] (-1, 1) circle[radius=2pt]
        (-1.5, 1.5) circle[radius=2pt]
        (-1.6, 1.6) circle[radius=2pt]
        (-2, 2) circle[radius=2pt]
        (-2.7, 2.7) circle[radius=2pt]
        (-3, 3) circle[radius=3pt];
    \filldraw[color=red] (1, -0.5) circle[radius=2pt]
        (1.4, -0.7) circle[radius=2pt]
        (1.8, -0.9) circle[radius=2pt]
        (2, -1) circle[radius=2pt]
        (2.5, -1.25) circle[radius=2pt]
        (3, -1.5) circle[radius=3pt];
    \end{tikzpicture}
    
    \caption{Hypothetical expected embeddings $U$ for DCBM in $\mathbb R^3$, where points represent rows of $U$ and colors represent blocks. (Circled points represent structure of expected SBM embeddings.)}
\end{figure}

This approach is theoretically justified by demonstrating that the observed embeddings $\hat U$ concentrate around a set of ``expected" embeddings $U$ that satisfy appropriate geometric properties. In particular, let $Y \sim \DCBM(\theta, \psi, B)$, let $P$ be the $n \times n$ matrix with entries $P_{ij} = \psi_i \psi_j B_{\theta_i \theta_j}$, and let $U$ be the $n \times k$ matrix that holds the leading $k$ eigenvectors of $P$. (Note that $E[Y] = P$ except on the diagonal.) In the case of SBM, it can be shown that the rows of $U$ correspond to $k$ distinct points, and two nodes belong to the same block if and only if their expected embeddings are equal. For general DCBM, the rows of $U$ fall along $k$ rays emanating from the origin, with each ray corresponding to a unique block \citep{jin2015fast,lei2015consistency}.

Unfortunately, these geometric properties are not, in general, preserved by the edge-flipping routine: the embeddings of $\Meps(Y)$ do not have the same expected geometric properties as the embeddings of $Y$. However, our Corollary~\ref{thm:scaledexpectation} suggests a modified approach: Since $E[(\depsMeps)(Y)]$ and $E[Y]$ share the same eigenvectors, our intuition suggests that the embeddings of $(\depsMeps)(Y)$ and $Y$ will concentrate in the same locations. Indeed, we prove such concentration bounds in Lemma~\ref{thm:frobenius}. This is also the key fact that powers the main results of Section~\ref{theoretical-results}. To perform spectral clustering on the edge-flipped network, then, we need only to modify the algorithms to use the downshifted adjacency matrix $(\depsMeps)(Y)$ in place of $\Meps(Y)$. Our proposed modified algorithms for SBM and DCBM are given in Algorithms~\ref{alg:spectral-kmeans} and \ref{alg:spectral-normalized-kmedians}, respectively. The $k$-means and $k$-medians problems, as well as their approximations, are defined in Section~\ref{k-means-k-medians-problems}.

\begin{algorithm}
    \caption{EF Spectral Clustering ($k$-means)}
    \label{alg:spectral-kmeans}
\begin{algorithmic}
   \STATE {\bfseries Input:} edge-flipped adjacency matrix $A \in \{0, 1\}^{n \times n}$, number of blocks $k$, approximation error $\gamma$, privacy budget $\eps$
   \STATE {\bfseries Output:} private estimate of block membership $\hat \theta \in [k]^n$
   
   \STATE {}
   \STATE Let $A_\downarrow = \begin{cases} A & \eps = \infty \\ A - (e^\eps + 1)^{-1}(\mathbf 1_n \mathbf 1_n^T - \mathbf I_n) & \eps < \infty \end{cases}$.
   \STATE Let $\hat u_1, \dots, \hat u_k \in \R^n$ be the $k$ leading eigenvectors (by absolute value) of $A_\downarrow$.
   \STATE Let $\hat \theta$ be a $(1 + \gamma)$-approximate solution to the $k$-means problem over the rows of $\hat U_\downarrow = [ \hat u_1  \; \dots \; \hat u_k ] \in \R^{n \times k}$.
   \STATE {\bfseries return} $\hat \theta$
\end{algorithmic}
\end{algorithm}

\begin{algorithm}
    \caption{EF Spectral Clustering (Normalized $k$-medians)}
    \label{alg:spectral-normalized-kmedians}
\begin{algorithmic}
   \STATE {\bfseries Input:} edge-flipped adjacency matrix $A \in \{0, 1\}^{n \times n}$, number of blocks $k$, approximation error $\gamma$, privacy budget $\eps$
   \STATE {\bfseries Output:} private estimate of block membership $\hat \theta \in [k]^n$
   
   \STATE {}
   \STATE Let $A_\downarrow = \begin{cases} A & \eps = \infty \\ A - (e^\eps + 1)^{-1}(\mathbf 1_n \mathbf 1_n^T - \mathbf I_n) & \eps < \infty \end{cases}$.
   \STATE Let $\hat u_1, \dots, \hat u_k \in \R^n$ be the $k$ leading eigenvectors (by absolute value) of $A_\downarrow$.
   \STATE Let $\hat U_\downarrow = [ \hat u_1 \; \dots \; \hat u_k ] \in \R^{n \times k}$.
   \STATE {}
   \STATE \textit{\# construct row-normalized matrix} $\hat U_\downarrow'$ \textit{over non-zero rows}
   \STATE Let $I_+ = \{ i \in [n] : \| (\hat U_\downarrow)_{i*} \|_2 > 0 \}$, $f : [|I_+|] \to [n]$ s.t. $f(i) = (I_+)_i$
   \FOR {$i = 1, \dots, |I_+|$}
        \STATE Let $(\hat U_\downarrow')_{i*} = (\hat U_\downarrow)_{f(i)*} / \| (\hat U_\downarrow)_{f(i)*} \|_2$
    \ENDFOR
   \STATE {}
   \STATE Let $\hat \theta'$ be a $(1 + \gamma)$-approximate solution to the $k$-medians problem over the rows of $\hat U_\downarrow'$.
   \STATE Let $\hat \theta_i = \begin{cases}
        1 & i \not \in I_+ \\
        \hat \theta'_{f^{-1}(i)} & i \in I_+
   \end{cases}, \quad i \in [n]$
   \STATE {\bfseries return} $\hat \theta$
\end{algorithmic}
\end{algorithm}

To evaluate the performance of the clustering algorithms, we assume knowledge of ground truth group labels $\theta$ and measure two forms of misclassification. The first is an overall measure of misclassification that captures the proportion of misidentified nodes, up to a permutation of labelings. Let $S_{[k]}$ denote the set of all permutations $\sigma: [k] \to [k]$, and let $\ind{\cdot}$ denote an indicator function. Then the \emph{overall misclassification} for the estimated groups $\hat \theta$ is given by:
\begin{equation}
\label{eq:misclassification}
L(\theta, \hat \theta) = \min_{\sigma \in S_{[k]}} \frac 1n \sum_{i=1}^n \ind{\sigma(\hat \theta_i) \neq \theta_i)}.
\end{equation}

As in \citet{lei2015consistency}, we also consider the \emph{worst-case misclassification} within a single community in order to account for trivialities\footnote{Consider an asymptotic regime with $k=2$ where one community is of constant size $N_1$ and the other is of size $n - N_1$. In this setting, a trivial estimator such as $\hat \theta = \mathbf 1_n$ can achieve consistency in the sense that $L(\theta, \hat \theta) \to 0$.}. This error is given by:
\begin{equation}
    \label{eq:wc-misclassification}
\tilde L(\theta, \hat \theta) = \max_{j \in [k]} \min_{\sigma \in S_{[k]}} \frac 1{|C_j|} \sum_{i \in C_j} \ind{\sigma(\hat \theta_i) \neq j}.
\end{equation}

%\section{Theoretical Results}
\section{Misclassification Bounds for Private Clustering}
\label{theoretical-results}

We generalize the results of \citet{lei2015consistency} developed for the non-private network setting to derive finite-sample misclassification bounds for spectral clustering of SBM and DCBM under local differential privacy using the modified spectral clustering algorithms proposed in Section~\ref{clustering-methods} (Algorithms~\ref{alg:spectral-kmeans} and \ref{alg:spectral-normalized-kmedians}). We argue in Section~\ref{clustering-methods} and prove in Section~\ref{proofs} that the spectral embeddings of $Y$ and $(\depsMeps)(Y)$ will concentrate in the same locations. In Lemma~\ref{thm:frobenius}, we also demonstrated that the concentration weakens with a stronger privacy guarantee (i.e., smaller $\eps$). In this section, we use these concentration results to derive misclassification bounds for private spectral clustering of SBM and DCBM; all proofs are provided in Section \ref{proofs}. We begin with the results for SBM.

\subsection{SBM}

\begin{theorem}
\label{thm:finitebounds}
Let $Y \sim \SBM(\theta, B)$ with $B$ full rank, $\max B \geq \log n/n$, and minimum absolute eigenvalue $\lambda_B > 0$. Let $\eps \in (0, \infty]$, and let $\hat \theta_\eps$ be the result of Algorithm~\ref{alg:spectral-kmeans} on $\M_\eps(Y)$ (where $\M_\infty(Y) = Y$), using an approximation error of $\gamma$. Let:
\begin{equation}
\label{eq:g-eps}
g_\eps(B) = \begin{cases}
    \frac{e^\eps + 1}{e^\eps - 1} \left(\max B + \frac{1}{e^\eps - 1}\right), & \eps < \infty \\
    \max B, & \eps = \infty
\end{cases}
\end{equation}
There exists a universal constant $c_1$ such that if:
\begin{equation}
\label{eq:sbm-condition}
\frac{(2 + \gamma) kn}{\nmin^2 \lambda_B^2} g_\eps(B) < c_1^{-1} ,
\end{equation}
then with probability at least $1 - n^{-1}$:
$$
\begin{aligned}
&\tilde L(\theta, \hat \theta_\eps) \leq c_1 \frac{(2 + \gamma) kn}{\nmin^2 \lambda_B^2} g_\eps(B) \\
\text{and}
\quad
& L(\theta, \hat \theta_\eps) \leq c_1 \frac{(2 + \gamma) kn_\text{max}'}{\nmin^2 \lambda_B^2} g_\eps(B) .
\end{aligned}
$$
\end{theorem}

Taking $\eps = \infty$ (i.e., the non-private setting) in Theorem~\ref{thm:finitebounds} recovers a result that is exactly equivalent to the results for SBM in \citet{lei2015consistency}. Consequently, the cost of privacy on misclassification is captured by $g_\eps(B)$. Noting that $g_\eps(B)$ is a decreasing function in $\eps$ and that $\lim_{\eps \to \infty} g_\eps(B) = g_\infty(B)$, Theorem~\ref{thm:finitebounds} provides a smooth interpolation between the known finite-sample misclassification bounds without privacy and the corresponding bounds under local DP.

We can use the finite-sample bounds from Theorem~\ref{thm:finitebounds} to construct particular asymptotic regimes under which $\hat \theta_\eps$ is consistent and derive convergence rates. For a given sequence of SBM networks, it suffices to show that:
$$
\frac{k n}{\nmin^2} g_\eps(B) = o(1) .
$$
If this holds, condition~(\ref{eq:sbm-condition}) is met for large $n$, and both measures of misclassification, $\tilde L(\theta, \hat \theta_\eps)$ and $L(\theta, \hat \theta_\eps)$, tend to zero. 

Since $g_\eps(B)$ explains the difference in classification performance between private and non-private clustering, a comparison of convergence rates can be boiled down to the differences in asymptotic behavior of $g_\eps(B)$. We note that if $\eps = \infty$, then $g_\eps(B) = \max B$. This will serve as our baseline. On the other hand, if $\eps < \infty$, then $g_\eps(B) = O(\max B + \zeta_\eps^{-1} + \zeta_\eps^{-2})$, where $\zeta_\eps = e^\eps - 1$, an increasing function of $\eps$ satisfying $\zeta_0 = 0$.  (See Fact~\ref{thm:g-inequality}.) Thus we have:
$$
\begin{aligned}
&\tilde L(\theta, \hat \theta_\eps) = \begin{cases}
    O_P \left( \frac{kn}{\nmin^2 \lambda_B^2} \max B \right) & \eps = \infty \\
    O_P \left( \frac{kn}{\nmin^2 \lambda_B^2} (\max B + \zeta_\eps^{-1} + \zeta_\eps^{-2}) \right) & \eps < \infty
\end{cases} \\
\text{and}
\quad
& L(\theta, \hat \theta_\eps) = \begin{cases}
    O_P \left( \frac{kn_\text{max}'}{\nmin^2 \lambda_B^2} \max B \right) & \eps = \infty \\
    O_P \left( \frac{kn_\text{max}'}{\nmin^2 \lambda_B^2} (\max B + \zeta_\eps^{-1} + \zeta_\eps^{-2}) \right) & \eps < \infty
\end{cases} .
\end{aligned}
$$

From this, it is clear that if $\eps$ and $\max B$ are bounded away from zero, we can obtain the same convergence rate under local DP as in the absence of privacy. For example, for a fixed privacy budget and a sequence of dense SBMs (i.e., one for which the edge probabilities do not tend to zero), we can effectively employ local DP with no cost to the convergence rate (up to a constant). For a sparse SBM (one in which $\max B \to 0$) with fixed privacy budget, however, this does not hold, as $g_\eps(B) = \Theta(1)$ for $\eps < \infty$, but $g_\infty(B) = \max B = o(1)$. Thus the convergence rate for sparse SBM slows by a factor of $(\max B)^{-1}$.

This slowing of convergence for sparse SBM limits the extent of sparsity that can be handled under this local DP mechanism. Consider a regime in which $k$ is fixed, communities grow proportionally ($\nmin = \Theta(n)$), and $B$ changes with $n$ only via some scaling parameter $\alpha_n \to 0$, i.e., $B = \alpha_n B_0$ for some fixed matrix $B_0$. In this case, we have $k$ constant, $\max B = \Theta(\alpha_n), \lambda_B = \Theta(\alpha_n)$, and so the worst-case misclassification is:
\begin{equation}
\label{eq:ex-sbm-scaled}
\tilde L(\theta, \hat \theta_\eps) = \begin{cases}
    O_P \left( \frac{1}{n \alpha_n} \right) & \eps = \infty \\
    O_P \left( \frac{\alpha_n + \zeta_\eps^{-1} + \zeta_\eps^{-2}}{n \alpha_n^2} \right) & \eps < \infty
\end{cases} .
\end{equation}
Theorem~\ref{thm:finitebounds} allows us to choose $\alpha_n$ as small as $\log n/n$, so we have consistency for non-private clustering down to $\log n/n$. For private clustering with a fixed privacy budget, however, we need $\alpha_n = \omega(n^{-1/2})$. The only way to achieve greater sparsity is to allow the privacy budget to grow arbitrarily large. For example, choosing $\eps = \log(1 + \alpha_n^{-1})$, we have $\zeta_\eps^{-1} = \alpha_n$, and thus we can recover the same convergence rate and accommodate the same level of sparsity as without privacy---but at the cost of allowing unbounded privacy loss for large $n$.

\subsection{DCBM}

Generalizing the theoretical results for SBM to DCBM requires a bit of care. As described in Section~\ref{privacy-mechanism}, in the more general setting of DCBM, the expected embeddings for a given block fall along distinct rays emanating from the origin---in contrast with distinct points in the case of SBM. It is for this reason that Algorithm~\ref{alg:spectral-normalized-kmedians} is used for DCBM. The key theorethical result from this more general treatment is given below.

\begin{theorem}
\label{thm:finiteboundsdcbm}
Let $Y \sim \DCBM(\theta, \psi, B)$ with $\max_{i \in C_j} \psi_i = 1$ for $j \in [k]$, $B$ full rank, $\max B \geq \log n/n$, and minimum absolute eigenvalue $\lambda_B > 0$. Let $\eps \in (0, \infty]$, and let $\hat \theta_\eps$ be the result of Algorithm~\ref{alg:spectral-normalized-kmedians} on $\M_\eps(Y)$ (where $\M_\infty(Y) = Y$), using an approximation error of $\gamma$. Let $g_\eps(B)$ as in Theorem~\ref{thm:finitebounds}. There exists a universal constant $c_2$ such that if:
\begin{equation}
\label{eq:dcbm-condition}
\frac{(2.5 + \gamma) \sqrt{kn \, g_\eps(B)}}{\tilde n_\text{min} \lambda_B} < c_2^{-1} \frac{\nmin}{\sqrt{\sum_{j = 1}^k n_j^2 \nu_j}},
\end{equation}
then with probability at least $1 - n^{-1}$:
$$
L(\theta, \hat \theta_\eps) \leq c_2 \frac{(2.5 + \gamma)}{\tilde n_\text{min} \lambda_B} \sqrt{\frac kn \left( \sum_{j = 1}^k n_j^2 \nu_j \right) g_\eps(B)}.
$$
\end{theorem}

As was the case with SBM, the results of Theorem~\ref{thm:finiteboundsdcbm} for $\eps = \infty$ (non-private setting) match the original results of \citet{lei2015consistency}. The cost of privacy is once again determined by the function $g_\eps(B)$. In this case, however, the bound changes with the square root of $g_\eps(B)$.

Comparing the results of Theorems~\ref{thm:finitebounds} and \ref{thm:finiteboundsdcbm}, the more general result involves additional parameters, of course, but also generally provides a weaker result. Ignoring constants, condition (\ref{eq:dcbm-condition}) is more stringent than (\ref{eq:sbm-condition}), as $\nmintilde \leq \nmin$, and $\nmin / \sqrt{ \sum_{j = 1}^k n_j^2 \nu_j} \leq k^{-1/2}$. The resulting upper bound on $L(\theta, \hat \theta_\eps)$ also results in weaker convergence rates. In keeping with \citet{lei2015consistency}, Theorem~\ref{thm:finiteboundsdcbm} offers no bound on $\tilde L(\theta, \hat \theta_\eps)$, as its proof bounds only the total number of misclassified nodes in the network. A trivial bound for $\tilde L$ can be obtained by observing that:
\begin{equation}
\label{eq:l-tilde-trivial}
\tilde L(\theta, \hat \theta_\eps) \leq (n/\nmin) L(\theta, \hat \theta_\eps) .
\end{equation}

To simplify matters, we illustrate an asymptotic regime in which communities grow proportionally ($n_j = \Theta(n/k)$ for all $j \in [k]$) and assume that there exists a global lower bound $0 < a \leq \psi_i$ for $i \in [n]$. Under these conditions, we can state a simple asymptotic result.
\begin{lemma}
\label{thm:dcbm-asymptotics}
Suppose $Y$ satisfies the conditions of Theorem~\ref{thm:finiteboundsdcbm}, and suppose further that $0 < a \leq \psi_i \leq 1$ for all $i \in [n]$ and $n_j = \Theta(n/k)$ for all $j \in [k]$. Then:
$$
\frac{k^2 \sqrt{g_\eps(B)}}{a^3 \lambda_B \sqrt n} = o(1) \implies L(\theta, \hat \theta_\eps) = O_P \left( \frac{k \sqrt{g_\eps(B)}}{a^3 \lambda_B \sqrt n} \right) .
$$
\end{lemma}

To compare this to the results seen in eq. (\ref{eq:ex-sbm-scaled}), consider again the case when $B = \alpha_n B_0$ for some fixed matrix $B_0$ and sequence $\alpha_n \to 0$. Then since $k$ is fixed, $n/\nmin = \Theta(1)$, and so combining Lemma~\ref{thm:dcbm-asymptotics} with eq. (\ref{eq:l-tilde-trivial}) yields a convergence rate of:
$$
\tilde L(\theta, \hat \theta_\eps) = \begin{cases}
    O_P \left( \frac{1}{a^3 \sqrt{n \alpha_n}} \right) & \eps = \infty \\
    O_P \left( \frac{\sqrt{\alpha_n + \zeta_\eps^{-1} + \zeta_\eps^{-2}}}{a^3 \alpha_n \sqrt n} \right) & \eps < \infty
\end{cases} .
$$
For SBM, where $a = 1$---or more generally for any regime in which $a$ is bounded away from zero---this bound is precisely the square root of the bound obtained in eq. (\ref{eq:ex-sbm-scaled}). Once again, we see that the convergence rate for sparse networks slows under private methods, and while some sparsity is attainable with privacy through this method, to accommodate the same level of sparsity with privacy as without, we would again need to allow the privacy-loss budget $\eps$ to grow arbitrarily large.

\section{Empirical Evaluations}
\label{empirical-evaluation}

\subsection{Simulation Studies}

The theoretical results above suggest that the edge-flip privacy mechanism can be used to achieve convergence for spectral clustering of SBM and DCBM networks that is similar in rate for dense networks, while slower for sparse networks. Here we evaluate these results empirically through simulation. To facilitate this analysis, we consider two special cases of the SBM and DCBM:

\begin{definition}
\label{def:ssbm}
The \textbf{symmetric SBM} is an SBM network consisting of $n$ nodes, $k$ equal-sized blocks, and $B = p \mathbf I_k + r \mathbf 1_k \mathbf 1_k^T$. For such a network, we write $Y \sim \SSBM(n, k, p, r)$.
\end{definition}

\begin{definition}
\label{def:sdcbm}
The \textbf{symmetric DCBM} is a DCBM network consisting of $n$ nodes, $k$ equal-sized blocks, $B = p \mathbf I_k + r \mathbf 1_k \mathbf 1_k^T$, and $\psi \in [a, 1]^n$ taking value 1 for the first node of a given block and values randomly drawn from $\mathrm{Uniform}(a, 1)$ elsewhere. For such a network, we write $Y \sim \SDCBM(n, k, p, r, a)$.
\end{definition}

Starting with SBM, we simulate two regimes following the symmetric SBM described in Definition~\ref{def:ssbm}. In the first setting, we use a dense symmetric SBM. We consider 16 values of $n$ ranging from $n = 30$ to $n = 12000$ and $\eps \in \{0.5, 0.75, 1, 1.5, 2, 3, 4, \infty\}$, where $\eps=\infty$ represents the original network (i.e., no privacy). For each $n, \eps$ pair, we draw 100 networks $Y^{(i)} \sim \SSBM(n, k = 3, p = 0.2, r = 0.05)$. We then apply Algorithm~\ref{alg:spectral-kmeans} to $\M_\eps(Y^{(i)})$, and report the mean overall misclassification rate $L(\theta^{(i)}, \hat \theta^{(i)})$ of eq. (\ref{eq:misclassification}). The results of this experiment are plotted on log--log axes in the left panel of Figure \ref{fig:sim-sbm}. Each curve, corresponding to a different value of $\eps$, appears to be approximately linear and match the rest in slope, suggesting that they have approximately equal polynomial rates of convergence. This is consistent with the results of Section~\ref{theoretical-results}.

For the second setting, we use a sparse symmetric SBM. Here we consider 13 values of $n$ ranging from $n = 10$ to $n = 12800$ and the same eight values of $\eps$ as in the first setting. For each $n, \eps$ pair, we draw 100 networks $Y^{(i)} \sim \SSBM(n, k = 2, p = 1.5 n^{-.3}, r = .15 n^{-.3})$, then apply Algorithm~\ref{alg:spectral-kmeans} to $\M_\eps(Y^{(i)})$, as in the first setting. The results are plotted on log--log axes in the right panel of Figure \ref{fig:sim-sbm}. Here, the curves corresponding to each value of $\eps$ are no longer parallel: for smaller values of $\eps$, the convergence rate slows.

\begin{figure}
\centering
\includegraphics[width=\textwidth]{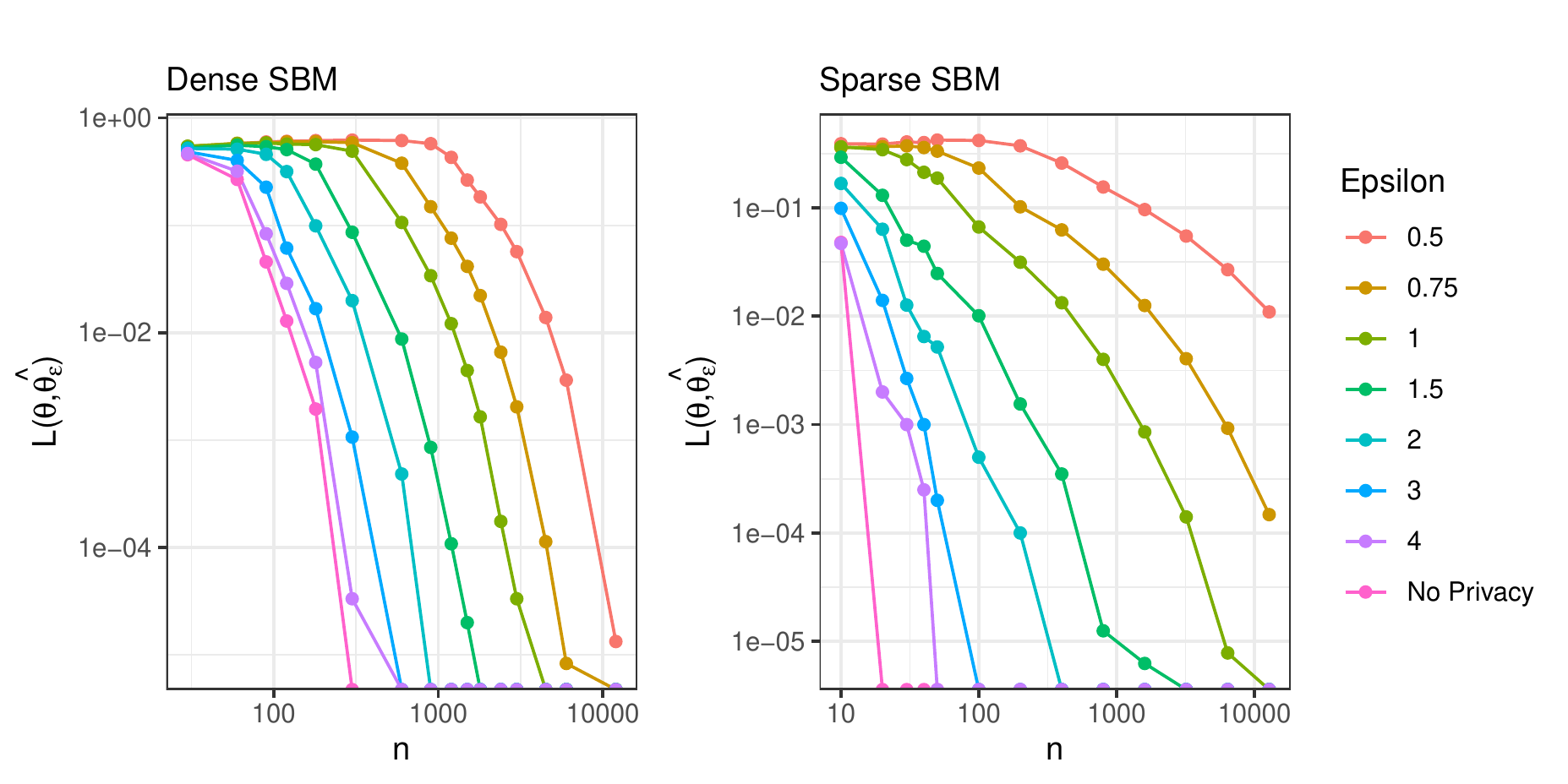}
\caption{Proportion of misclassified nodes in simulated SBM networks for various values of $\eps, n$. Left (dense): $Y \sim \SSBM(n, k = 3, p = 0.2, r = 0.05)$, Right (sparse): $Y \sim \SSBM(n, k = 2, p = 1.5 n^{-.3}, r = .15 n^{-.3})$}
\label{fig:sim-sbm}
\end{figure}

For DCBM, we simulate two additional regimes, one dense and one sparse, in a manner similar to the SBM simulations. In the dense setting, we use a dense symmetric DCBM (Definition~\ref{def:sdcbm}) with the same values of $n, \eps$ as used in the dense SBM simulations. For each $n, \eps$ pair, we draw 100 networks $Y^{(i)} \sim \SDCBM(n, k = 3, p = 0.4, r = 0.05, a = 0.3)$. In the sparse setting, we use the same $\eps$ values as in the other settings and 12 values of $n$ ranging from $n = 20$ to $n = 12800$. Then we draw 100 networks $Y^{(i)} \sim \SDCBM(n, k = 2, p = 2n^{-.25}, r = 0.1n^{-.25}, a = 0.3)$. In both DCBM settings, we report the mean overall misclassification rate after applying Algorithm~\ref{alg:spectral-normalized-kmedians}. The results of these simulations are plotted on log--log axes in Figure~\ref{fig:sim-dcbm}. In contrast with the SBM simulations, the DCBM simulations generally exhibit slower convergence. Similar to the SBM simulations, we see a clear slowing of convergence rate for smaller values of $\eps$ in the sparse setting.

\begin{figure}
\centering
\includegraphics[width=\textwidth]{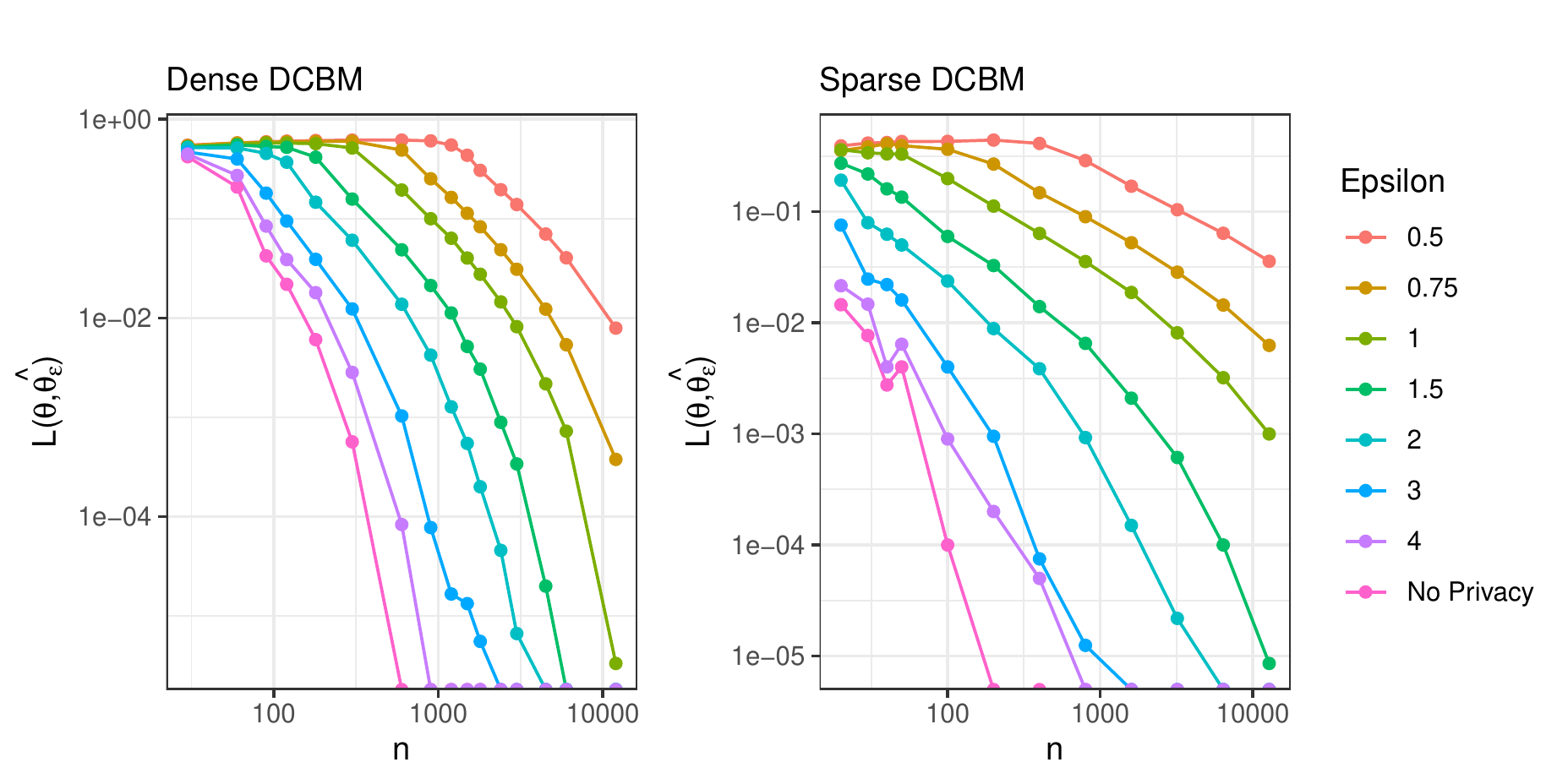}
\caption{Proportion of misclassified nodes in simulated DCBM networks for various values of $\eps, n$. Left (dense): $Y \sim \SDCBM(n, k = 3, p = 0.4, r = 0.05, a = 0.3)$, Right (sparse): $Y \sim \SDCBM(n, k = 2, p = 2n^{-.25}, r = 0.1n^{-.25}, a = 0.3)$}
\label{fig:sim-dcbm}
\end{figure}

While the results of all these simulations are consistent with the theory of Section~\ref{theoretical-results}, the precise nature in which convergence rate slows with $\eps$ in a sparse regime is not fully captured by the theoretical results. Indeed, for larger values of $\eps$, the convergence properties are similar to the non-private setting. In general, based on these simulations, it seems the convergence bounds given here and in \citet{lei2015consistency} are not tight, in the sense that the observed convergence rates appear to beat the theoretical guarantees, sometimes by considerable order.

\subsection{Performance on Observed Networks}

To assess the practicality of these methods, we applied the edge-flip mechanism and Algorithm~\ref{alg:spectral-normalized-kmedians} to several real-world datasets, then compared the performance of spectral clustering on the private networks over various privacy budgets. We show the trade-offs of privacy loss and accuracy.

\begin{figure}
\centering
\includegraphics[width=.7\textwidth]{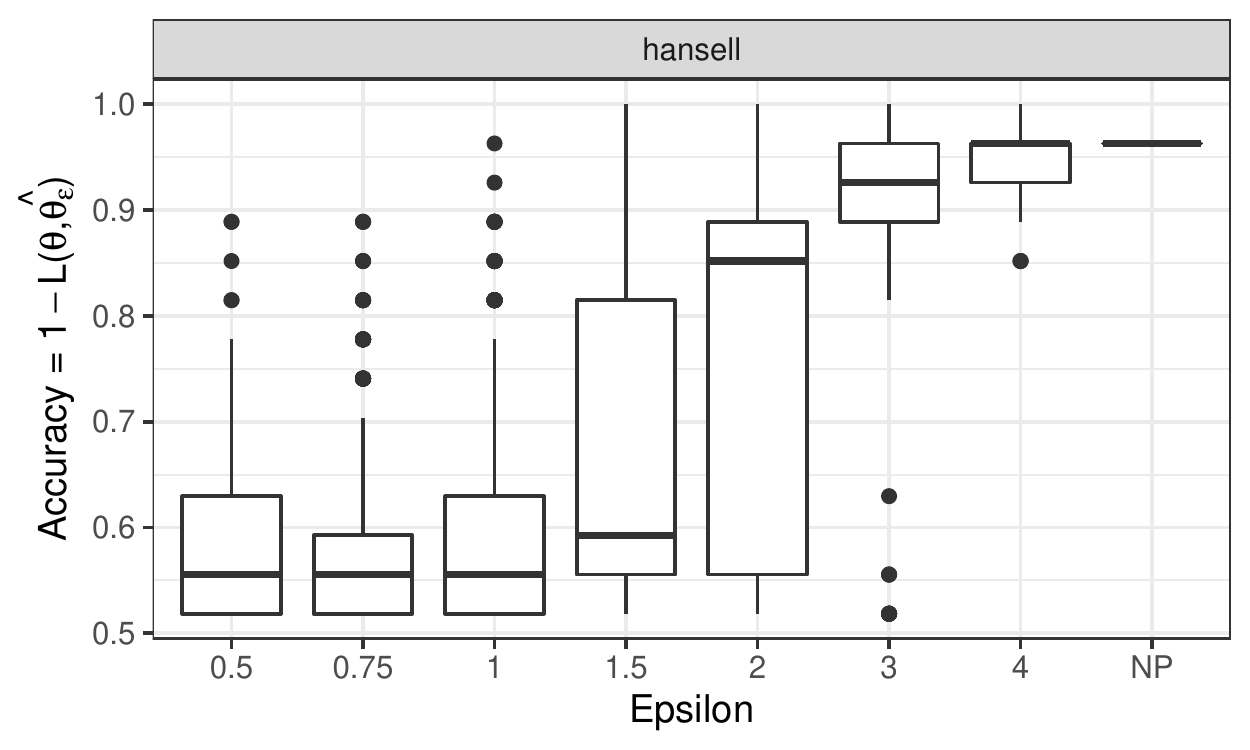}
\caption{Trade-off of privacy-loss and accuracy of private spectral clustering on a small dataset, Hansell's friendship data \citep{hansell1984cooperative}, with varying privacy-loss $\eps$, and ``NP" denoting non-private data.}
\label{fig:datasets_small}
\end{figure}

The first dataset we considered is a small, classical network used in the SBM literature: Hansell's friendship data \citep{hansell1984cooperative}, which was used in \citet{wang1987stochastic,snijders_estimation_1997}. This is a directed network of $n=27$ (13 male, 14 female) students, and the presence of an edge $(i, j)$ indicates that student $i$ considers student $j$ to be a friend. We symmetrized the network by setting $Y_{ij} = \max \{ Y_{ij}, Y_{ji} \}$, and we used the students' sexes as ground-truth labels.

For this small network, we considered $\eps \in \{0.5, 0.75, 1, 1.5, 2, 3, 4, \infty \}$, and for each value of $\eps$, we applied the privacy mechanism and spectral clustering algorithm 500 times. Box plots showing the overall classification accuracy, $1 - L(\theta, \hat \theta)$ from eq. \ref{eq:misclassification}, for each value of $\eps$ are given in Figure \ref{fig:datasets_small}. Accuracy remains high for values of $\eps > 2$; for such a small network, this is about as good as we can expect.

\begin{figure}
\centering
\includegraphics[width=\textwidth]{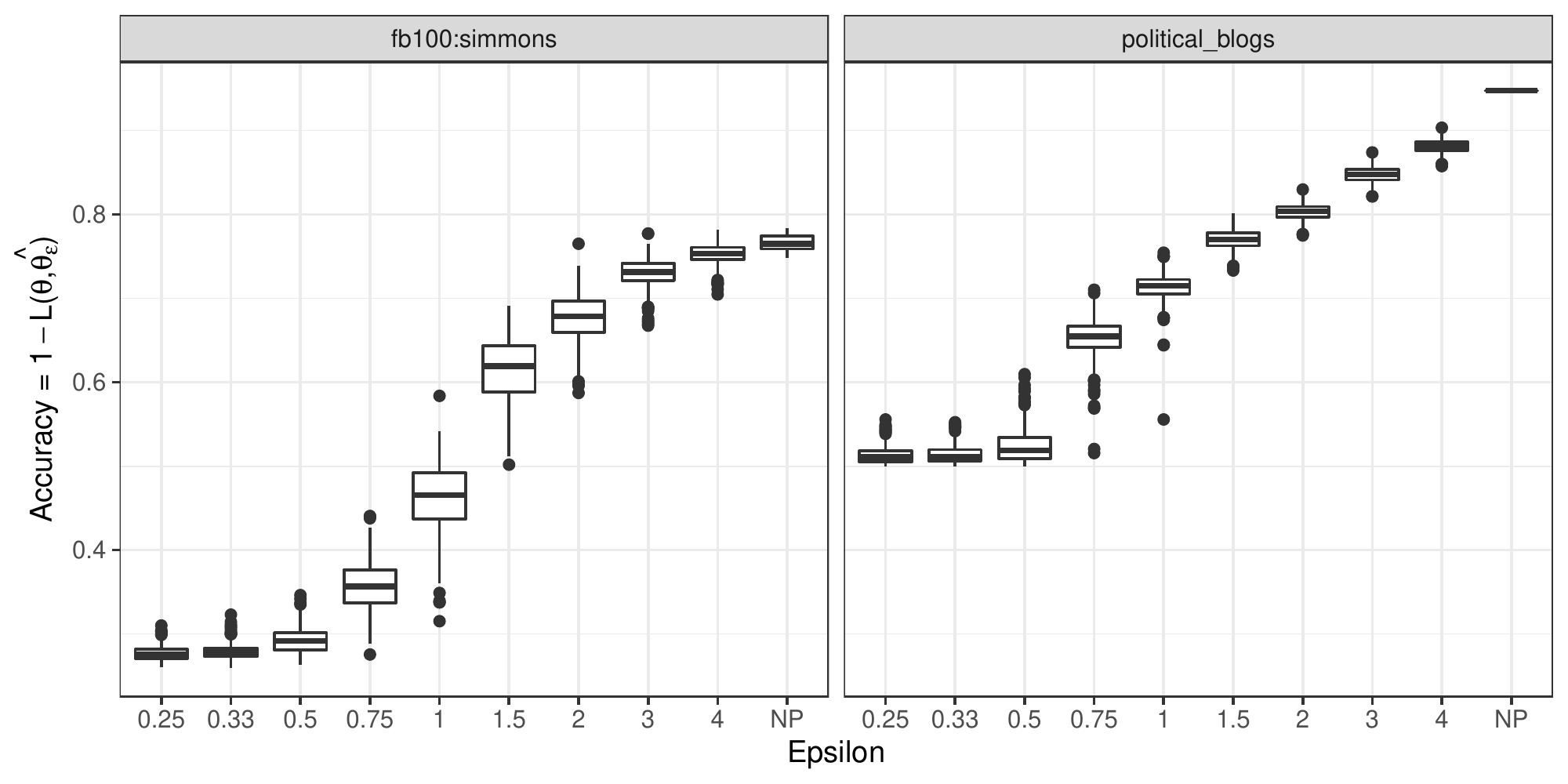}
\caption{Trade-off of privacy-loss and accuracy of private spectral clustering on well-known DCBM datasets, with varying privacy-loss $\eps$, and ``NP" denoting non-private data. Left: Facebook friendships of Simmons College students \citep{traud2012social}, Right: political blogs network \citep{adamic2005political}}
\label{fig:datasets_dcbm}
\end{figure}

Next, we considered two datasets from the DCBM literature. The first network consists of $n=1137$ Facebook users from Simmons College \citep{traud2012social,chen2018convexified}, and edges represent friendships. Each student belongs to one of $k=4$ class years (2006 to 2009). These class years are used as ground-truth group memberships, which are inferred through their friendships. The second dataset is the well-known network of political blogs from \citet{adamic2005political}, which has been widely used in the DCBM literature \citep{karrer2011stochastic,jin2015fast,chen2018convexified,abbe_community_2018}. This network consists of $n=1222$ blogs on U.S. politics, which have been categorized as either left-leaning or right-leaning ($k=2$ groups). An edge in this network represents the existence of a hyperlink between the two blogs, and these hyperlinks are used to infer the left- or right-leaning label for each node. The performance of the private spectral clustering methods on these networks is shown in Figure~\ref{fig:datasets_dcbm}. In both of these datasets, we see a steady decline in performance as $\eps$ decreases, without a clear inflection point.

\begin{figure}
\centering
\includegraphics[width=.8\textwidth]{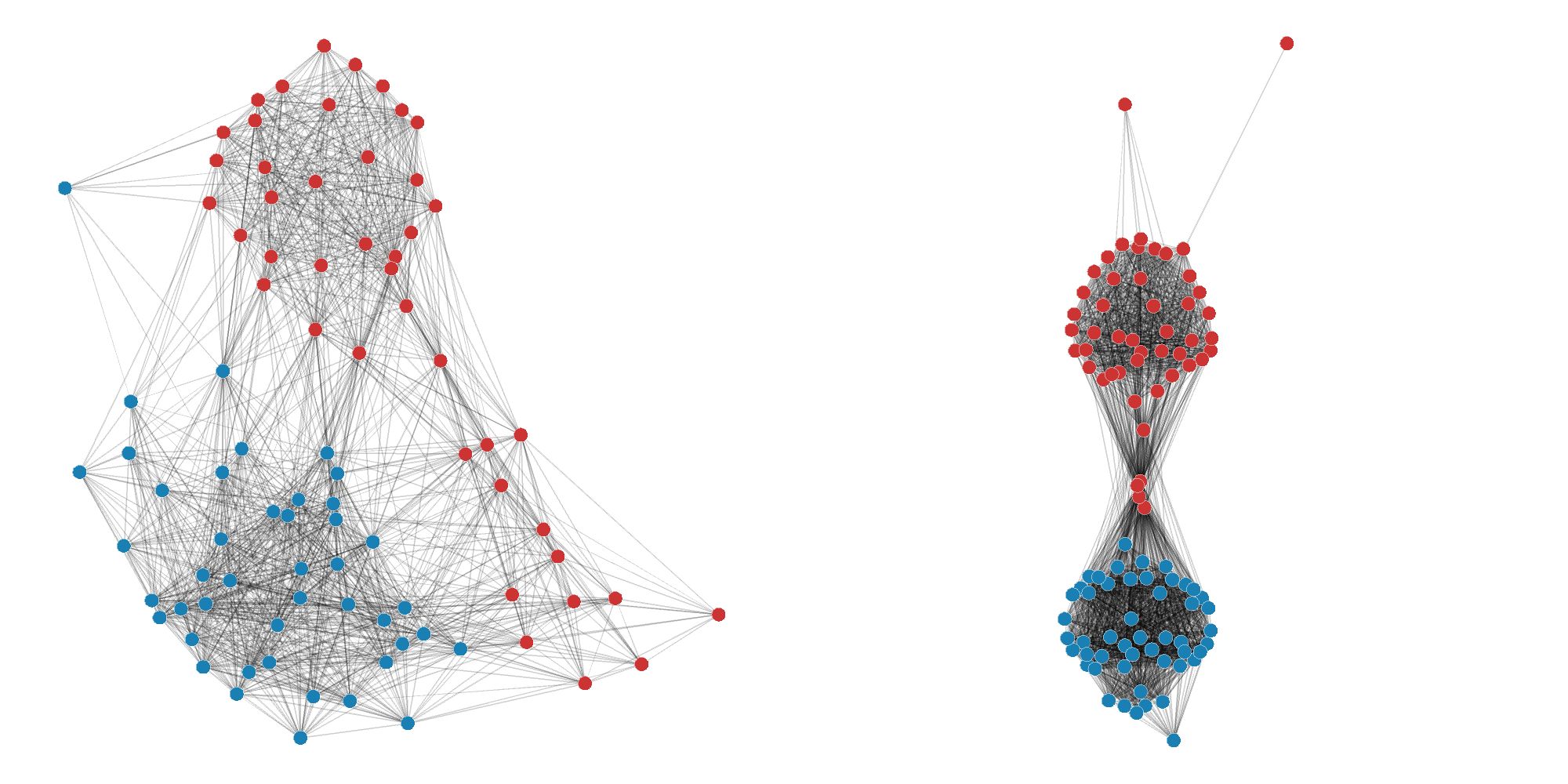}
\caption{Visualization of congressional voting networks for 70th U.S. Senate (left) and 110th U.S. Senate (right). Democrats are depicted as blue nodes and Republican as red.}
\label{fig:networks_senate}
\end{figure}

The theoretical and simulated results from earlier suggest that these methods are best paired with suitably large or dense networks. A relevant collection of networks can be found, for example, in \citet{andris2015rise}, where networks of Democrat and Republican members of the U.S. House of Representatives are constructed based on voting similarity. Based on their methods and data from \citet{voteview2021}, we reconstructed these networks for a number of sessions of the U.S. House of Representatives and U.S. Senate.\footnote{Although these networks are produced from public information, networks involving public figures like congresspeople illustrate a case where edge privacy is particularly relevant, since the set of nodes is public knowledge, but their interactions or relations may be sensitive.}

Changes over time in the patterns of voting among congresspeople offer us a range of networks that span nearly complete separation of the parties (particularly in recent years) to networks that exhibit greater levels of connectivity across parties. Figure \ref{fig:networks_senate} depicts two such networks, with the 70th U.S. Senate on the left and the 110th U.S. Senate on the right. In the 70th Senate, 13\% of cross-party pairs are connected, while 76\% and 59\% of Democrat and Republican pairs are connected (respectively). In the 110th Senate, 9\% of cross-party pairs are connected, while 98\% and 88\% of Democrat and Republican pairs are connected (respectively).

\begin{figure}
\centering
\includegraphics[width=\textwidth]{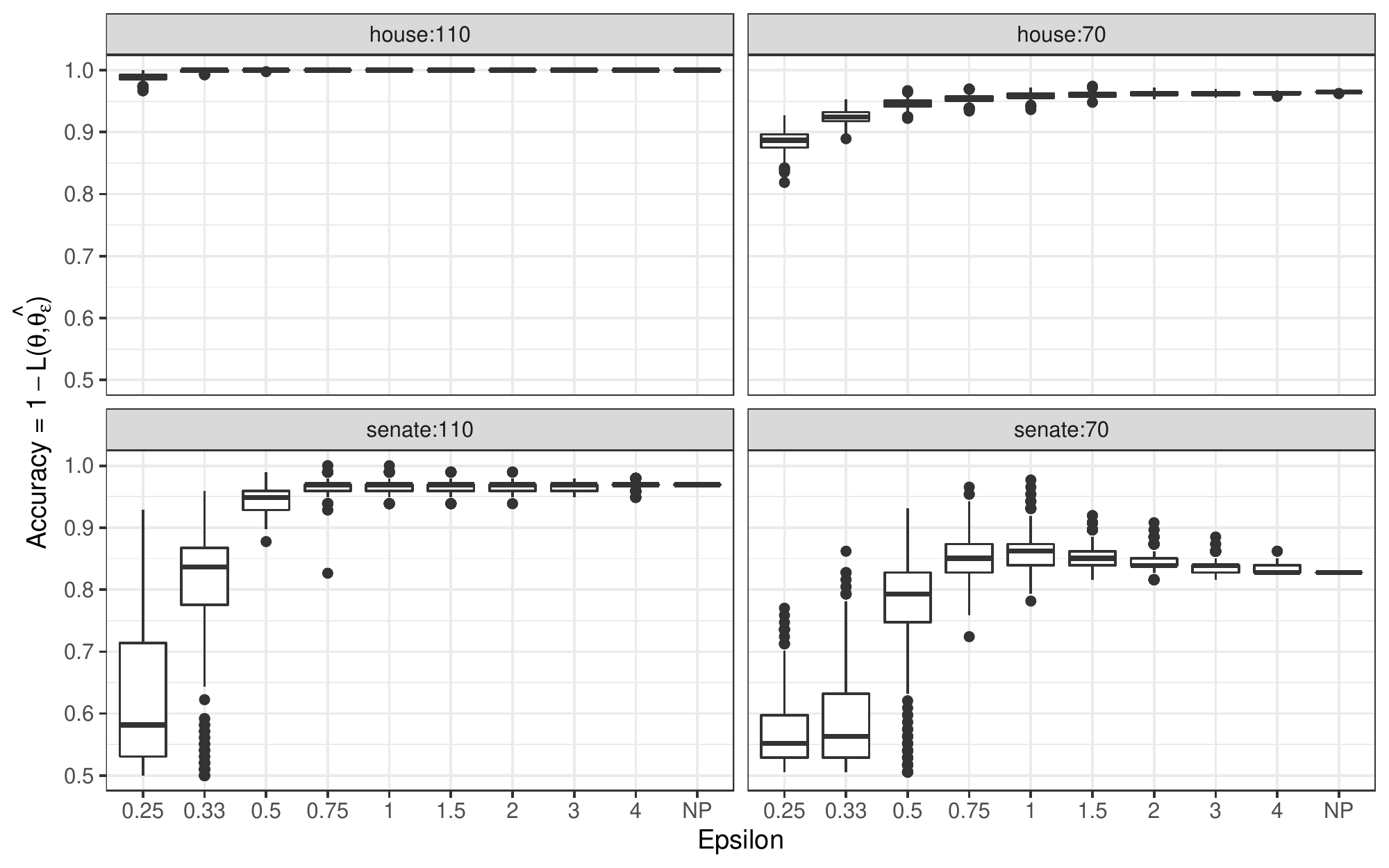}
\caption{Trade-off of privacy-loss and accuracy of private spectral clustering on U.S. Congressional voting networks. Clockwise from top left: 110th House of Representatives, 70th House of Representatives, 70th Senate, 110th Senate}
\label{fig:datasets_congress}
\end{figure}

For the 70th Senate ($n=87$), 110th Senate ($n=98$), 70th House ($n=425$), and 110th House ($n=423$), we applied the private spectral clustering methods 50 times for each value of $\eps \in \{0.25, 0.33, 0.5, 0.75, 1, 1.5, 2, 3, 4, \infty \}$. The results are depicted in Figure \ref{fig:datasets_congress}. The high density and relatively low variability of degree in these networks enables impressive clustering performance: for the House networks, clustering performance is solid for all $\eps$ considered; in the smaller Senate networks, we see a drop in performance for values of $\eps < 0.5$.

\section{Discussion}
\label{discussion}

Just as SBM and DCBM are fundamental network models featuring community structure, the edge-flip mechanism is a fundamental network privacy mechanism providing local differential privacy on networks at the edge level. A great advantage of the edge-flip mechanism is that it produces a synthetic network with clear distributional properties. In particular, the explicit mixture distribution of $\M_\eps(Y)$ described in Lemma~\ref{thm:mixture}---in contrast with the unclear distributional properties of some alternative mechanisms---greatly facilitates the process of accounting for privacy in various statistical procedures. Moreover, by nature of producing a synthetic network, any number of statistical procedures can be applied to the single output $\M_\eps(Y)$ with a fixed privacy budget $\eps$.

Our work demonstrates how to use the clear distribution of $\M_\eps(Y)$ to extend the theory of consistent community detection to the local DP setting. This is, to our knowledge, the first attempt to address this important problem. Algorithms~\ref{alg:spectral-kmeans} and \ref{alg:spectral-normalized-kmedians} represent modest extensions to the algorithms originally proposed in \citet{lei2015consistency}, requiring only the addition of a small ``downshift" transformation to unlock a powerful privacy guarantee with a clear accounting of the cost to clustering performance. In fact, Lemma~\ref{thm:closure} suggests that for SBM networks in particular, this downshift transformation is optional, as $\M_\eps(Y)$ is itself an SBM whose community structure is identical to $Y$.

Focusing on the specific case of SBM, we may ask whether it is preferable to perform clustering for SBM on $(\depsMeps)(Y)$, as prescribed by Algorithm~\ref{alg:spectral-kmeans}, or to proceed with ordinary spectral clustering on $\M_\eps(Y)$ per the closure property of Lemma~\ref{thm:closure}. By combining the closure property with the $\eps = \infty$ case of Theorem~\ref{thm:finitebounds}, one can obtain performance bounds in terms of the parameter $\tau_\eps(B)$. In some cases, these bounds may be superior to the bounds resulting from Theorem~\ref{thm:finitebounds} alone.\footnote{For example, if $B$ is positive-definite, then one can show that $g_\infty(\tau_\eps(B)) / \lambda_{\tau_\eps(B)}^2 < g_\eps(B) / \lambda_B^2$ via Weyl's inequality.} Nonetheless, we have opted to present the results for the downshifted procedure for the sake of generality. This allows us to keep the SBM misclassification bounds in terms of $\lambda_B$ instead of $\lambda_{\tau_\eps(B)}$, where a general relationship between these two quantities is elusive.

The benefits of the edge-flip mechanism's tractable distribution are, of course, not limited to the network models considered in this paper, nor to the field of network clustering. For example, \citet{karwa_sharing_2017} have previously shown how to adjust inferential procedures for exponential random graph models using the edge-flip mechanism. These examples are likely only scratching the surface of the theory that can be extended to accommodate privacy under the edge-flip mechanism.

The primary drawback to the edge-flip mechanism is its relationship to sparse networks. We posit that this limitation is inherent to working with a local-DP synthetic network, in which the magnitude of noise required for privacy begins to dwarf the signal present in a sparse network. Looking at the edge-flip mechanism in particular, for a sparse network $Y$, the mixture $\M_\eps(Y)$ will be considerably more dense, and a disproportionate number of its edges will result from the privacy mechanism, not the original network of interest. Although these changes preserve key properties of SBM and DCBM networks, reductions in sparsity dilute the signal-to-noise ratio considerably. While the theory supports the claim that these methods are consistent for SBM and DCBM networks exhibiting mild sparsity, empirical performance for sparse networks indicates a considerable utility loss when a very strong privacy guarantee is applied to a sparse network. Other privacy mechanisms preserving the sparsity of a network, especially non-local mechanisms, could be considered for greater performance in these cases.

The trade-off of privacy-loss and accuracy that we observe for label recovery on dense networks such as the Congressional voting networks, where considerable accuracy is maintained at $\eps = 0.5$, suggests that the methods employed here are of more than just theoretical interest. By comparison, in their work with exponential random graph models and the edge-flip mechanism, \citet{karwa_sharing_2017} use values of $\eps$ ranging from 3 to 6 on a network of similar size. Nonetheless, it would be useful to develop theoretical results for alternative privacy mechanisms that can accommodate greater sparsity.

Lastly, while the results given here focus on spectral clustering of SBM and DCBM, some of the insights from these techniques are suggestive of further theoretical developments in related fields. For example, concentration bounds for the spectral embeddings of the downshifted private network could likely be extended in the more general family of random dot product graphs \citep{athreya2017statistical}. This opens a wide range of avenues for future work.

\section{Proofs of Main Results}
\label{proofs}

The proofs of the results from Section~\ref{theoretical-results} closely follow the techniques of \citet{lei2015consistency}. For both SBM and DCBM, we will require the following notation. Let $\eig_k(\cdot) : \R^{n \times n} \to \R^{n \times k}$ denote a function that returns the eigenvectors of its argument corresponding to the $k$ largest eigenvalues in absolute value.

Let $Y \sim \DCBM(\theta, \psi, B)$, where $\psi = 1$ for SBM. Let $P$ be the $n \times n$ matrix with entries $P_{ij} = \psi_i \psi_j B_{\theta_i \theta_j}$. Note that $E[Y] = P - \diag(P)$, as $Y$ has zeros on the diagonal by construction. Let $\hat U = \eig_k(Y)$ denote the eigenvectors of the observed non-private adjacency matrix $Y$, and $U = \eig_k(P)$ denote the eigenvectors of the expected adjacency matrix (plus a diagonal).

We will use a $\downarrow$ subscript to denote private analogs to these quantities. Assuming $\eps < \infty$, let $Y_\downarrow = (\depsMeps)(Y)$ denote the edge-flipped and downshifted synthetic network. Since $E[Y_\downarrow] = \frac{e^\eps - 1}{e^\eps + 1} E[Y]$ (Corollary~\ref{thm:scaledexpectation}), we let $P_\downarrow = \frac{e^\eps - 1}{e^\eps + 1} P$, and let $\hat U_\downarrow = \eig_k(Y_\downarrow)$ and $U_\downarrow = \eig_k(P_\downarrow)$. If $\eps = \infty$, we have no privacy, so we let $Y_\downarrow = Y, P_\downarrow = P, \hat U_\downarrow = \hat U, U_\downarrow = U$.

A critical fact for the proofs of the main theorems is that the matrix $U_\downarrow$ can be chosen to be precisely equal to $U$, even if $\eps < \infty$, since if $V \Lambda V^T$ is an eigendecomposition of $P$, then $V \left( \frac{e^\eps - 1}{e^\eps + 1} \Lambda \right) V^T$ is an eigendecomposition of $P_\downarrow$. For this reason, we will assume below that $U_\downarrow = U$.

We begin with a lemma that captures the main technical distinction between our results and those of \citet{lei2015consistency}.

\begin{lemma}
\label{thm:frobenius}
Suppose $Y \sim \DCBM(\theta, \psi, B)$ with $\max_{i \in C_j} \psi_i = 1$ for $j \in [k]$, $B$ full rank, $\max B \geq \log n/n$, and minimum absolute eigenvalue $\lambda_B > 0$. Let $\eps \in (0, \infty]$, and let $g_\eps(B)$ as defined in eq. (\ref{eq:g-eps}). There exists a universal constant $c_0$ and a $k \times k$ orthogonal matrix $Q$ such that:
$$
P \left( \| \hat U_\downarrow - U_\downarrow Q \|_F \leq c_0 \frac{2 \sqrt{2kn g_\eps(B)}}{\nmintilde \lambda_B} \right) \geq 1 - n^{-1} .
$$
\end{lemma}

\begin{proof}
By Lemma~5.1 of \citet{lei2015consistency}, there exists a $k \times k$ orthogonal matrix $Q$ such that:
$$
\begin{aligned}
\| \hat U_\downarrow - U_\downarrow Q \|_F &\leq \frac{2 \sqrt{2k}}{\lambda_{P_\downarrow}} \| Y_\downarrow - P_\downarrow \| \\
&= \frac{2 \sqrt{2k}}{\lambda_{P_\downarrow}} \| \Meps(Y) - E [\Meps(Y)] - \diag(P_\downarrow) \| \\
&\leq \frac{2 \sqrt{2k}}{\lambda_{P_\downarrow}} \left( \| \Meps(Y) - E [\Meps(Y)] \| + 1 \right)
\end{aligned}
$$
where the last inequality follows from the fact that $\| \diag(P_\downarrow) \| \leq 1$, as every entry of $P_\downarrow$ is bounded in the unit interval. From here, it remains to find a lower bound for $\lambda_{P_\downarrow}$ and an upper bound for $\| \Meps(Y) - E [\Meps(Y)] \|$.

We begin with the simpler task of bounding $\lambda_{P_\downarrow}$. Consider first the case when $\eps = \infty$. In this case, $P_\downarrow = P$. From the proof of \citet{lei2015consistency} Lemma~4.1, we know that the nonzero eigenvalues of $P$ are precisely those of $\Psi B \Psi$, where $\Psi = \diag( \sqrt{\tilde n_1}, \dots, \sqrt{\tilde n_k})$. Since both $B$ and $\Psi$ are symmetric and invertible, we can say that:
$$
\begin{aligned}
\| (\Psi B \Psi)^{-1} \|
& \leq \| \Psi^{-1} \| \| B^{-1} \| \| \Psi^{-1} \| \\
& = \lambda_\Psi^{-1} \lambda_B^{-1} \lambda_\Psi^{-1} \\
& = \nmintilde^{-1} \lambda_B^{-1} .
\end{aligned}
$$
The fact that $\| (\Psi B \Psi)^{-1} \|$ is the largest absolute eigenvalue of $(\Psi B \Psi)^{-1}$ further implies that the smallest absolute eigenvalue of $\Psi B \Psi$ satisfies:
$$
\lambda_P = \lambda_{\Psi B \Psi} \geq \nmintilde \lambda_B .
$$
When $\eps < \infty$, recall that $P_\downarrow = \frac{e^\eps - 1}{e^\eps + 1} P$, so:
$$
\lambda_{P_\downarrow} \geq \begin{cases}
    \nmintilde \lambda_B, & \eps = \infty \\
    \frac{e^\eps - 1}{e^\eps + 1} \nmintilde \lambda_B, & \eps < \infty
\end{cases} .
$$

Next, we can upper bound $\| \Meps(Y) - E [\Meps(Y)] \|$ using Theorem~5.2 of \citet{lei2015consistency}. First, we establish necessary bounds on the entries of $E[\Meps(Y)]$. Let:
$$
\mu = \begin{cases}
    \max B, & \eps = \infty \\
    \frac{e^\eps - 1}{e^\eps + 1} \left( \max B \right) + \frac 2{e^\eps + 1} \left(\frac 12 \right), & \eps < \infty
\end{cases} .
$$
In both cases, we can see that $\mu$ is an upper bound for $\max E[\Meps(Y)]$, per the definition of DCBM and Lemma~\ref{thm:mixture}. Additionally, we can view $\mu$ as an affine combination of $\max B$ and $\frac 12$, from which it is clear that $\mu \geq \min \{\frac 12, \max B \} \geq \log n/n$ (for $n \geq 1$). So by \citet{lei2015consistency} Theorem~5.2, with probability at least $1 - n^{-1}$:
$$
\| \Meps(Y) - E[\Meps(Y)] \| \leq C \sqrt{n \mu} ,
$$
where $C = C(1, 1)$ for $C(\cdot, \cdot)$ defined in \citet{lei2015consistency}. Combining some facts, we now have that with probability at least $1 - n^{-1}$:
$$
\| \hat U_\downarrow - U_\downarrow Q \|_F \leq 2 \sqrt{2k} \frac{C \sqrt{n \mu} + 1}{\lambda_{P_\downarrow}}.
$$
Since $\mu \geq \log n/n$, then $1 \leq (\log 2)^{-1/2} \sqrt{n \mu}$ (for $n > 1$). Thus if we let $c_0 = C + (\log 2)^{-1/2}$, we can simplify the above to:
$$
\| \hat U_\downarrow - U_\downarrow Q \|_F \leq c_0 \frac{2 \sqrt{2 k n \mu}}{\lambda_{P_\downarrow}}.
$$

Finally, it remains to put these results in terms of $g_\eps(B)$ and $\lambda_B$. When $\eps = \infty$, it is clear that:
$$
\frac{\sqrt{\mu}}{\lambda_{P_\downarrow}} \leq \frac{\sqrt{\max B}}{\nmintilde \lambda_B} = \frac{\sqrt{g_\eps(B)}}{\nmintilde \lambda_B} .
$$
When $\eps < \infty$, this same inequality holds, albeit with different intermediate steps:
$$
\begin{aligned}
\frac{\sqrt{\mu}}{\lambda_{P_\downarrow}}
& \leq \frac 1{\nmintilde \lambda_B} \frac{e^\eps + 1}{e^\eps - 1} \sqrt{ \frac{(e^\eps - 1) \max B + 1}{e^\eps + 1}} \\
& = \frac 1{\nmintilde \lambda_B} \sqrt{ \frac{e^\eps + 1}{e^\eps - 1} \left( \max B + \frac 1{e^\eps - 1} \right)} \\
& \leq \frac {\sqrt{g_\eps(B)}}{\nmintilde \lambda_B} .
\end{aligned}
$$
Thus, we conclude that with probability at least $1 - n^{-1}$:
$$
\| \hat U_\downarrow - U_\downarrow Q \|_F \leq c_0 \frac{2 \sqrt{2kn g_\eps(B)}}{\nmintilde \lambda_B} .
$$
\end{proof}

\subsection{The $k$-means and $k$-medians Problems}
\label{k-means-k-medians-problems}

Recall that Algorithm~\ref{alg:spectral-kmeans} requires solving an approximate $k$-means problem, while Algorithm~\ref{alg:spectral-normalized-kmedians} requires solving an approximate $k$-medians problem. The proofs of Theorems~\ref{thm:finitebounds} and \ref{thm:finiteboundsdcbm} rely on some additional notation and properties surrounding these problems.

Let $\mathbb G^{n \times k}$ denote the set of $n \times k$ membership matrices consisting of a single one in each row with zeroes elsewhere. In both the $k$-means and $k$-medians problems applied to the rows of the matrix $\hat U$, we seek a membership matrix $\Theta \in \mathbb G^{n \times k}$ and set of centroids $X \in \mathbb R^{k \times k}$ that minimize the distance $\| \Theta X - \hat U \|_*$ for a suitable norm $\| \cdot \|_*$. In $k$-means, the norm chosen is the Frobenius norm, while in the $k$-medians problem, the norm chosen is the (2,1) norm, $\| A \|_{2,1} = \sum_{i} \| A_{i*} \|_2$. Finding an exact solution for each of these problems is difficult, but efficient approximation algorithms exist. We will call a solution $\hat \Theta \hat X$ a $(1 + \gamma)$-approximate solution to the $k$-means problem if:
$$
\| \hat \Theta \hat X - \hat U \|_F^2 \leq (1+\gamma) \left[ \inf_{\Theta' \in \mathbb G^{n,k}, X'_{k \times k}} \| \Theta' X' - \hat U \|_F^2 \right] .
$$
Similarly, we will call $\hat \Theta \hat X$ a $(1 + \gamma)$-approximate solution to the $k$-medians problem if:
\begin{equation}
\label{eq:k-medians}
\| \hat \Theta \hat X - \hat U \|_{2,1} \leq (1+\gamma) \left[ \inf_{\Theta' \in \mathbb G^{n,k}, X'_{k \times k}} \| \Theta' X' - \hat U \|_{2,1} \right] .
\end{equation}
In the main text, we denote the membership parameter $\theta$ as a vector. One can easily convert a membership matrix $\Theta \in \mathbb G^{n \times k}$ to a membership vector $\theta \in [k]^n$ by choosing $\theta_i = \arg \max_j \Theta_{ij}$.

\subsection{SBM}

\begin{proof}[Proof of Theorem~\ref{thm:finitebounds}]
We begin by recalling that $\SBM(\theta, B) \overset D= \DCBM(\theta, \mathbf 1_n, B)$ (\ref{eq:sbm-dcbm-relation}), and so by Lemma~\ref{thm:frobenius}, there exists a universal constant $c_0$ and orthogonal matrix $Q$ such that with probability at least $1 - n^{-1}$:
\begin{equation}
\label{sbm-proof-condition}
\| \hat U_\downarrow - U_\downarrow Q \|_F \leq c_0 \frac{2 \sqrt{2kn g_\eps(B)}}{\nmin \lambda_B} .
\end{equation}

From here, the proof follows the same line of argument as \citet{lei2015consistency} Theorem~3.1 and Corollary~3.2. For completeness, and since our parameterization is a bit different, we include the remainder of the proof here.

Since $U_\downarrow = U$, we know from \citet{lei2015consistency} Lemma~3.1 that  $U_\downarrow Q = \Theta X'$ for some $\Theta \in \mathbb G^{n,k}$ and $X'_{k \times k}$ satisfying:
$$
\| X'_{j*} - X'_{\ell*} \|_2 = \sqrt{n_j^{-1} + n_\ell^{-1}}
$$
for $j \neq \ell$. Now we wish to apply Lemma~5.3 of \citet{lei2015consistency}. For $j \in [k]$, choose $\delta_j = \sqrt{ n_j^{-1} + (\max_{\ell \neq j} n_\ell)^{-1} }$, and define $S_j$ as in Lemma~5.3 of \citet{lei2015consistency}. We want to show that $(16 + 8 \gamma) \| \hat U_\downarrow - U_\downarrow Q \|_F^2 / \delta_j^2 < n_j$ for $j \in [k]$. Since $\delta_j^2 n_j > 1$ for all $j$, it suffices to show that $(16 + 8 \gamma) \| \hat U_\downarrow - U_\downarrow Q \|_F^2 \leq 1$. Under the condition that (\ref{sbm-proof-condition}) holds, we have that:
$$
(16 + 8 \gamma) \| \hat U_\downarrow - U_\downarrow Q \|_F^2 \leq \frac{64 c_0^2 (2 + \gamma) kn g_\eps(B)}{\nmin^2 \lambda_B^2}
$$
Choosing $c_1 = 64 c_0^2$, then (\ref{eq:sbm-condition}) implies that the above is $\leq 1$, as desired. Thus, for each community $j \in [k]$, the set of nodes that are possibly misclassified by Algorithm~\ref{alg:spectral-kmeans} must be a subset of $S_j$, and since $\delta_j^2 > n_j^{-1}$:
$$
\begin{aligned}
\sum_{j = 1}^k \frac{|S_j|}{n_j}
\leq \sum_{j = 1}^k |S_j| \delta_j^2
&\leq 4(4+2\gamma) \| \hat U_\downarrow - U_\downarrow Q \|_F^2 \\
&\leq 64 c_0^2 \frac{(2 + \gamma) kn g_\eps(B)}{\nmin^2 \lambda_B^2} \\
&= c_1 \frac{(2 + \gamma) kn g_\eps(B)}{\nmin^2 \lambda_B^2} .
\end{aligned}
$$
So then we have that:
$$
\begin{aligned}
\tilde L(\theta, \hat \theta_\eps) &\leq \max_{j \in k} \frac{|S_j|}{n_j} \leq c_1 \frac{(2 + \gamma) kn}{\nmin^2 \lambda_B^2} g_\eps(B) \\
\quad L(\theta, \hat \theta_\eps) &\leq \frac 1n \sum_{j = 1}^k |S_j| \leq c_1 \frac{(2 + \gamma) k \nmax'}{\nmin^2 \lambda_B^2} g_\eps(B) .
\end{aligned}
$$
\end{proof}

\subsection{DCBM}

\begin{proof}[Proof of Theorem~\ref{thm:finiteboundsdcbm}]
We begin with the results of Lemma~\ref{thm:frobenius}, which states that there exists a universal constant $c_0$ and a $k \times k$ orthogonal matrix $Q$ such that with probability at least $1 - n^{-1}$:
$$
\| \hat U_\downarrow - U_\downarrow Q \|_F \leq c_0 \frac{2 \sqrt{2kn g_\eps(B)}}{\nmintilde \lambda_B}.
$$

The remainder of the proof follows very closely from the proofs of Lemma~A.1 and Theorem~4.2 from \citet{lei2015consistency}. We define the sets $I_0$, $I_+$, and $S$ as in the original proofs. Recall that Algorithm~\ref{alg:spectral-normalized-kmedians} requires separate handling of zero vs. non-zero rows of $\hat U_\downarrow$. The set $I_0 = \{ i \in [n] \mid (\hat U_\downarrow)_{i*} = 0 \}$ holds the set of nodes whose embeddings are zero, and $I_+ = [n] \setminus I_0$ holds the remainder. We allow that the nodes in $I_0$ will be misclassified, and we define a set $S \subseteq I_+$ in which we also allow misclassification. Our goal, then, is to bound $|I_0 \cup S|$. We start by bounding $|I_0|$. Note that:
$$
\begin{aligned}
\| \hat U_\downarrow - U_\downarrow Q \|_F^2 &\geq \sum_{i \in I_0} \| (U_\downarrow Q)_{i*} \|_2^2 \\
&\geq \frac{|I_0|^2}{\sum_{i \in I_0} \|(U_\downarrow Q)_{i*}\|_2^{-2}} \quad \text{(Cauchy-Schwarz)}\\
&\geq \frac{|I_0|^2}{\sum_{i = 1}^n \|U_{i*}\|_2^{-2}} \quad \text( U_\downarrow = U, Q \text{ orthogonal)} \\
\end{aligned}
$$
From \citet{lei2015consistency} Lemma 4.1, we know that $\| U_{i*} \|_2^{-2} = \tilde n_{\theta_i} \psi_i^{-2}$, so $\sum_{i = 1}^n \|U_{i*}\|_2^{-2} = \sum_{j=1}^k n_j^2 \nu_j$. Thus:
$$
|I_0| \leq \| \hat U_\downarrow - U_\downarrow Q \|_F \sqrt{\sum_{j=1}^k n_j^2 \nu_j}.
$$

Moving on to $I_+$, we construct $\hat U'_\downarrow$ of size $|I_+| \times k$ to be the row-normalized version of $\hat U_\downarrow$, excluding zero rows, i.e., for $i = 1, \dots, |I_+|$ and $j = (I_+)_i$:
$$
(\hat U'_\downarrow)_{i*} = (\hat U_\downarrow)_{j*} / \|(\hat U_\downarrow)_{j*}\|_2 .
$$
Let $U_\downarrow'$ be the $n \times k$ row-normalized version of the expected embeddings $U_\downarrow$, where zero rows in $U_\downarrow$ are preserved as zero rows in $U_\downarrow'$. Then let $U_\downarrow''$ to be the $|I_+| \times k$ matrix constructed from the non-zero rows of $U_\downarrow'$, i.e., for $i = 1, \dots, |I_+|$ and $j = (I_+)_i$:
$$
(U_\downarrow'')_{i*} = (U_\downarrow')_{j*}.
$$
By (\ref{eq:k-medians}), if $\hat \Theta_+ \hat X$ is a $(1+\gamma)$-approximate solution to the $k$-medians problem on $\hat U_\downarrow$, then $\| \hat \Theta_+ \hat X - \hat U_\downarrow' \|_{2,1} \leq (1+\gamma) \| \hat U_\downarrow' - U_\downarrow'' Q \|_{2,1}$, and so:
$$
\begin{aligned}
\| \hat \Theta_+ \hat X - U_\downarrow'' Q \|_{2,1} &\leq \| \hat \Theta_+ \hat X - \hat U_\downarrow' \|_{2,1} + \| \hat U_\downarrow' - U_\downarrow'' Q \|_{2,1} \\
& \leq (2 + \gamma) \| \hat U_\downarrow' - U_\downarrow'' Q \|_{2,1}
\end{aligned}
$$
Using the fact that, for any vectors $v_1, v_2$ of equal dimension, $\| \frac{v_1}{\|v_1\|} - \frac{v_2}{\|v_2\|} \| \leq 2 \frac{ \| v_1 - v_2 \|}{\|v_1\|}$ \citep{lei2015consistency}, we can bound the above (2,1) norm as follows:
$$
\begin{aligned}
\| \hat U_\downarrow' - U_\downarrow'' Q \|_{2,1} &= \sum_{i \in I_+} \| (\hat U'_\downarrow)_{i*} - (U_\downarrow'' Q)_{i*} \|_2 \quad \text{(norm definition)} \\
&\leq 2 \sum_{i = 1}^n \frac{ \| (\hat U_\downarrow)_{i*} - (U_\downarrow Q)_{i*} \|_2 }{\| (U_\downarrow Q)_{i*} \|_2} \quad \text{(fact above)} \\
&= 2 \sqrt{ \left( \sum_{i = 1}^n \| (\hat U_\downarrow)_{i*} - (U_\downarrow Q)_{i*} \|_2 \| (U_\downarrow)_{i*} \|_2^{-1} \right)^2} \\
&\leq 2 \sqrt{ \left( \sum_{i = 1}^n \| (\hat U_\downarrow)_{i*} - (U_\downarrow Q)_{i*} \|_2^2 \right) \left( \sum_{i=1}^n \| (U_\downarrow)_{i*} \|_2^{-2} \right)} \quad \text{(Cauchy-Schwarz)} \\
&= 2 \sqrt{ \| \hat U_\downarrow - U_\downarrow Q \|_F^2 \left( \sum_{i=1}^n \| U_{i*} \|_2^{-2} \right)} \quad \text( U_\downarrow = U \text) \\
&= 2 \| \hat U_\downarrow - U_\downarrow Q \|_F \sqrt{\sum_{j=1}^k n_j^2 \nu_j} \\
\end{aligned}
$$

Let $S = \{ i \in I_+ : \| (\hat \Theta_+)_{i*} \hat X - (U_\downarrow'' Q)_{i*} \|_2 \geq 2^{-1/2} \}$. These are the nodes in $I_+$ for which the corresponding fitted cluster centroids are located more than $2^{-1/2}$ from their expectation (up to the transformation $Q$). Observe:
$$
2^{-1/2}|S| \leq \sum_{i \in I_+} \| (\hat \Theta_+)_{i*} \hat X - (U_\downarrow'' Q)_{i*} \|_2 = \| \hat \Theta_+ \hat X - U_\downarrow'' Q \|_{2,1}
$$
Thus we can bound $|S|$:
$$
\begin{aligned}
|S| &\leq \sqrt 2 \| \hat \Theta_+ \hat X - U_\downarrow'' Q \|_{2,1} \\
&\leq \sqrt 2 (2 + \gamma) \| \hat U_\downarrow' - U_\downarrow'' Q \|_{2,1} \\
&\leq 2 \sqrt 2 (2 + \gamma) \| \hat U_\downarrow - U_\downarrow Q \|_F \sqrt{\sum_{j=1}^k n_j^2 \nu_j}
\end{aligned}
$$
And so we can bound $|I_0 \cup S|$:
$$
|I_0 \cup S| = |I_0| + |S| \leq [1 + 2 \sqrt 2 (2 + \gamma)] \, \| \hat U_\downarrow - U_\downarrow Q \|_F \sqrt{\sum_{j=1}^k n_j^2 \nu_j}
$$
With probability $1 - n^{-1}$, this is further bounded:
$$
\begin{aligned}
|I_0 \cup S| &\leq [1 + 2 \sqrt 2 (2 + \gamma)] c_0 \frac{2 \sqrt{2kn g_\eps(B)}}{\nmintilde \lambda_B} \sqrt{ \sum_{j=1}^k n_j^2 \nu_j} \\
&\leq 8c_0 (2.5 + \gamma) \frac{\sqrt{kn g_\eps(B)}}{\nmintilde \lambda_B} \sqrt{ \sum_{j=1}^k n_j^2 \nu_j}
\end{aligned}
$$
Under the condition that the above is less than $n_\text{min}$, which is implied by (\ref{eq:dcbm-condition}), then for each $j \in [k]$, there must exist $i \in (C_j \setminus (I_0 \cup S))$---i.e., every block has at least one node whose normalized spectral embedding is within $2^{-1/2}$ of its expectation (in $\ell_2$ norm). Because the expected cluster centers $U''_{i*}$ are orthogonal and have unit norm, we have $\| U''_{i*} - U''_{j*} \|_2 = \sqrt 2$ whenever $\theta_i \neq \theta_j$ (or 0 otherwise). Thus for any two nodes $i, j \not \in I_0 \cup S$, we have:
$$
\begin{aligned}
(\hat \Theta_+)_{i*} = (\hat \Theta_+)_{j*} &\implies \| U''_{i*} - U''_{j*} \|_2 \\
& \quad \quad \quad \leq  \| (\hat \Theta_+)_{i*} \hat X - U'_{i*} Q \|_2 + \| (\hat \Theta_+)_{j*} \hat X - U'_{j*} Q \|_2 \\
& \quad \quad \quad < \sqrt  2 \\
&\implies U_{i*}'' = U_{j*}'' \\
&\implies \theta_i = \theta_j
\end{aligned}
$$
In other words, the set of nodes that are misclassified must be a subset of $I_0 \cup S$, and so $L(\theta, \hat \theta_\eps) \leq \frac{|I_0 \cup S|}{n}$. The final theorem statement is obtained by choosing $c_2 = 8 c_0$.
\end{proof}

\begin{proof}[Proof of Lemma~\ref{thm:dcbm-asymptotics}]
We begin by stating a few key facts under the assumed conditions:
$$
\begin{aligned}
\nu_j &= \left( \frac 1{n_j} \sum_{i \in C_j} \psi_i^2 \right) \left( \frac 1{n_j} \sum_{i \in C_j} \psi_i^{-2} \right) \in [a^2, a^{-2}] \\
\nmintilde &= \min_{j \in [k]} \sum_{i \in C_j} \psi_i^2 \in [a^2 \nmin, \nmin] \\
\nmin &= \Theta(n/k)
\end{aligned}
$$

From here, we apply Theorem~\ref{thm:finiteboundsdcbm}. First, we want to show that under the assumed conditions, (\ref{eq:dcbm-condition}) is satisfied for large $n$. A lower bound for the RHS of (\ref{eq:dcbm-condition}) is:
$$
%\begin{aligned}
c_2^{-1} \frac{\nmin}{\sqrt{ \sum_{j=1}^k n_j^2 \nu_j }} 
= \Omega \left( \frac{n/k}{\sqrt{ k (n/k)^2 a^{-2}}} \right) %\\
= \Omega \left( k^{-1/2} a \right)
%\end{aligned}
$$
Then looking at the LHS of (\ref{eq:dcbm-condition}), we can write:
$$
%\begin{aligned}
\frac{(2.5 + \gamma) \sqrt{kn \, g_\eps(B)}}{\tilde n_\text{min} \lambda_B}
= O \left( \frac{\sqrt{kn \, g_\eps(B)}}{a^2(n/k) \lambda(B)} \right) %\\
= O \left( \frac{k^{3/2} \sqrt{g_\eps(B)}}{a^2 \lambda(B) \sqrt n} \right)
%\end{aligned}
$$
Thus Theorem~\ref{thm:finiteboundsdcbm} is applicable for large $n$ if:
$$
\frac{k^{3/2} \sqrt{g_\eps(B)}}{a^2 \lambda(B) \sqrt n} = o \left( k^{-1/2} a \right)
$$
which is satisfied by assumption. Thus we conclude that with probability $1 - n^{-1}$:
$$
\begin{aligned}
L(\theta, \hat \theta_\eps)
&\leq c_2 \frac{(2.5 + \gamma)}{\tilde n_\text{min} \lambda_B} \sqrt{\frac kn \left( \sum_{j = 1}^k n_j^2 \nu_j \right) g_\eps(B)} \\
&= O \left( \frac{1}{a^2 (n/k) \lambda_B} \sqrt{\frac kn (k(n/k)^2 a^{-2}) g_\eps(B)} \right) \\
&= O \left( \frac{k \sqrt{g_\eps(B)}}{a^3 \lambda_B \sqrt{n}} \right) .
\end{aligned}
$$
\end{proof}

\subsection{Miscellaneous}

Below we prove a useful inequality for $g_\eps(B)$ claimed in Section~\ref{theoretical-results}.

\begin{fact}
\label{thm:g-inequality}
Let $k \in \mathbb N, \eps \in (0, \infty), B \in [0, 1]^{k \times k}$. Then:
$$
g_\infty(B) < g_\eps(B) \leq \max B + 3 \zeta_\eps^{-1} + 2 \zeta_\eps^{-2}
$$
where $\zeta_\eps = e^\eps - 1$.
\end{fact}
\begin{proof}
Note first that:
$$
\frac{e^\eps + 1}{e^\eps - 1} = \frac{e^\eps - 1}{e^\eps - 1} + \frac{2}{e^\eps - 1}  = 1 + 2 \zeta_\eps^{-1} .
$$
Thus:
$$
\begin{aligned}
g_\eps(B) & = \frac{e^\eps + 1}{e^\eps - 1} \left( \max B + \frac{1}{e^\eps - 1} \right) \\
& = (1 + 2 \zeta_\eps^{-1})(\max B + \zeta_\eps^{-1}) \\
& = \max B + \zeta_\eps^{-1} + 2 \zeta_\eps^{-1} \max B + 2 \zeta_\eps^{-2} \\
& \leq \max B + 3 \zeta_\eps^{-1} + 2 \zeta_\eps^{-2}
\end{aligned}
$$
where the last line follows from the fact that $\max B \leq 1$. The fact that $g_\infty(B) < g_\eps(B)$ also follows from the above, as $\zeta_\eps$ (and thus $\zeta_\eps^{-1}$, $\zeta_\eps^{-2}$) is strictly positive.
\end{proof}

\subsection{Privacy Guarantee}
\label{privacy-proof}

Finally, for completeness, we conclude with a formal proof of the privacy guarantee.

\begin{proof}[Proof of Theorem~\ref{thm:privacy}]
Let $Y, Y'$ be two binary, undirected networks of $n$ nodes differing on one edge, $(i, j)$ with $1 \leq i < j \leq n$. Then for all $k \neq i$, we have that $\M_k(Y_{k*}) \overset D = \M_k(Y'_{k*})$. Thus, all that remains to show is that for any $a \in \{0, 1\}^n$:
$$
P(\M_i(Y_{i*}) = a) \leq e^\eps P(\M_i(Y'_{i*}) = a).
$$
To simplify notation, we will write $\M_i$ in another form. Let:
$$
\M'(x) = \begin{cases}
    1 - x & \text{w.p.}\quad  \frac{1}{1 + e^\eps} \\
    x & \text{w.p.} \quad \frac{e^\eps}{1 + e^\eps}
\end{cases} .
$$
Then we can write:
$$
\M_i(Y_{i*}) = \left( \underbrace{0, \dots, 0}_{i \text{ entries}}, \; \M'(Y_{i,i+1}), \dots, \M'(Y_{in}) \right) ,
$$
where the $\M'(\cdot)$ are taken independently. Since the first $i$ entries of $\M_i(Y_{i*})$ are deterministic, we can safely restrict our attention to the case when $a_k = 0$ for $k \leq i$ (since otherwise we have $P(\M_i(Y_{i*}) = a) = P(\M_i(Y'_{i*}) = a) = 0$). Thus:
$$
\begin{aligned}
\frac{P(\M_i(Y_{i*}) = a)}{P(\M_i(Y'_{i*}) = a)}
&= \frac{\prod_{\ell = i+1}^n P(\M'(Y_{i\ell}) = a_\ell)}{\prod_{\ell = i+1}^n P(\M'(Y_{i\ell}) = a_\ell)} \\
&= \frac{P(\M'(Y_{ij}) = a_j)}{P(\M'(Y'_{ij}) = a_j)} \\
&\leq \frac{P(\M'(1) = 1)}{P(\M'(0) = 1)} \\
&= e^\eps ,
\end{aligned}
$$
where the inequality is taken by considering all combinations of $Y_{ij}$, $Y'_{ij}$, and $a_j$.
\end{proof}

\section*{Acknowledgements}

Computations for this research were performed on the Pennsylvania State University's Institute for Computational and Data Sciences' Roar supercomputer. This work was supported in part by NSF Award No. SES-1853209 to The Pennsylvania State University.

\bibliography{sources}
\bibliographystyle{agsm}

\end{document}